\documentclass[10pt]{amsart}
\usepackage{epsfig}
\usepackage{color}
\usepackage{amsmath, amsthm, amssymb, euscript, enumerate}
\usepackage{pxfonts}
\usepackage{graphicx}
\usepackage{subfigure}
\usepackage{mathtools}
\usepackage{ulem}
\usepackage{appendix}
\usepackage[pdftex]{hyperref}

\newcommand{\mbb}[1]{\mathbb{#1}}
\newcommand{\mc}[1]{\mathcal{#1}}

\newcommand{\R}{\mathbb{R}}
\newcommand{\Z}{\mathbb{Z}}

\newcommand{\e}{\varepsilon}
\newcommand{\vp}{\varphi}

\newcommand{\p}{\partial}
\newcommand{\ra}{\rightarrow}
\newcommand{\lra}{\longrightarrow}

\newcommand{\ov}[1]{\overline{#1}}

\newcommand\norm[1]{\left\Arrowvert #1\right\Arrowvert}

\newcommand{\osc}{\operatornamewithlimits{osc}}

\newcommand{\diam}{\operatorname{diam}}

\newtheorem{thm}{Theorem}[subsection]
\newtheorem{lem}{Lemma}[subsection]
\newtheorem{defn}{Definition}[subsection]

\newtheorem*{rem}{Remark}
\newtheorem{prop}{Proposition}[subsection]

\theoremstyle{definition}
\newtheorem*{acknowledgements}{Acknowledgements}

\numberwithin{equation}{subsection}

\title[Higher Order Converegence Rates]{Higher Order Convergence Rates in Theory of Homogenization I: Equations of Non-divergence Form}
\author{ Sunghan Kim}
\address{Department of Mathematical Sciences, Seoul National University, Seoul 151-747, Korea }
\email{ sunghan290@snu.ac.kr}
\thanks {2000 Mathematics Subject Classification: 35J15, 35J66,
74Q05, 74Q15, 74Q24}
\author{Ki-Ahm Lee}
\address{ Department of Mathematical Sciences, Seoul National University, Seoul 151-747, Korea,
          \&
             Center for Mathematical Challenges, Korea Institute for Advanced Study, Seoul 130-722, Korea}
\email{kiahm@snu.ac.kr}

\begin{document}

\begin{abstract}
We establish higher order convergence rates in the theory of periodic homogenization of both linear and fully nonlinear uniformly elliptic equations of non-divergence form. The rates are achieved by involving higher order correctors which fix the errors occurring both in the interior and on the boundary layer of our physical domain. The proof is based on a viscosity method and a new regularity theory which captures the stability of the correctors with respect to the shape of our limit profile.\end{abstract}
\maketitle
\tableofcontents

\section{Introduction}
\label{sec:intro}

We establish higher order convergence rates in the theory of periodic homogenization of both linear and fully nonlinear uniformly elliptic equations of non-divergence form. It is known that the equations containing highly oscillating variables $\frac{x}{\e}$, where the oscillation takes place periodically in the microscopic scale, exhibit a limiting behavior as $\e\ra 0$. More precisely, for the following $\e$-problems with linear operators,
\begin{equation}\tag{$L_\e$}\label{eq:l}
\begin{cases}
a_{ij}\left(\frac{x}{\e}\right)D_{ij}u^\e=f&\text{in }\Omega,\\
u^\e=g&\text{on }\p\Omega,
\end{cases}
\end{equation}
the solutions $u^\e$ converge to a function $u$ as $\e\ra 0$, which solves a boundary value problem
\begin{equation}\tag{$\ov{L}$}\label{eq:l-h}
\begin{cases}
\ov{a}_{ij}D_{ij}u=f&\text{in }\Omega,\\
u=g&\text{on }\p\Omega,
\end{cases}
\end{equation}
whose operator is homogenous (i.e., the matrix $(\bar{a}_{ij})$ is constant) with respect to the enviroment. For more details, one may refer to \cite{BLP} and \cite{JKO}. A similar behavior also exists when the operator consists of nonlinearity, namely,
\begin{equation}\tag{$F_\e$}\label{eq:nl}
\begin{cases}
F\left(D^2u^\e,x,\frac{x}{\e}\right)=0&\text{in }\Omega,\\
u^\e=g&\text{on }\p\Omega.
\end{cases}
\end{equation}
As in the linear case, the solutions $u^\e$ exhibit a limiting behavior, and the limit profile $u$ turns out to be a solution of the following PDE,
\begin{equation}\tag{$\ov{F}$}\label{eq:nl-h}
\begin{cases}
\ov{F}(D^2u,x)=0&\text{in }\Omega,\\
u=g&\text{on }\p\Omega.
\end{cases}
\end{equation}
where $\ov{F}$ is no longer oscillatory in the microscopic scale. For more details, see \cite{E2}.

In this paper, we give a quantitative analysis on the rate of convergence between the solution $u^\e$ and its limit profile $u$, and we further accelerate the rate by involving appropriate corrector functions for both interior and boundary layer of the physical domain. Finally we end up with a rigorous justification of the following two scale expansion of the solution $u^\e$:
\begin{equation}\label{eq:exp}
u^\e(x)=u(x)+\e(w_1^\e(x)+z_1^\e(x))+\cdots+\e^m(w_m^\e(x)+z_m^\e(x))+O(\e^{m-1}),
\end{equation} 
where $w_k^\e$ and $z_k^\e$ are the $k$-th order correctors which fix the error occurring in the interior and on the boundary layer respectively, and $m$ is the positive integer related to the regularity of the operator of the $\e$-problem. The above expression is explicit if the $\e$-problem is linear, but rather implicit when a nonlinearity comes in. We make a remark that our result is true also for operators with lower order dependence; essentially most of the challenges lie in proving the case for \eqref{eq:l} and \eqref{eq:nl} while the desired extensions and generalizations are fairly straightforward to obtain.

\subsection{Linear equations}\label{sec:l-intro}

Set $\Omega$ to be a bounded domain in $\mbb{R}^n$ with $C^{m+2,\alpha}$ boundary and let $f\in C^{m,\alpha}(\ov{\Omega})$ and $g\in C^{m+2,\alpha}(\ov{\Omega})$ for some exponent $0<\alpha\leq 1$ and an integer $m\geq 2$. We suppose that $A(y)=(a_{ij}(y))$, $1\leq i,j\leq n$ is a symmetric matrix-valued function defined in $\mbb{R}^n$ satisfying the following hypotheses: 

\begin{enumerate}
\item[(L1)] (Periodicity) $A(y+k)=A(y)$;
\item[(L2)] (Uniform Ellipticity) $\lambda|\xi|^2\leq a_{ij}(y)\xi_i\xi_j\leq\Lambda |\xi|^2$;
\item[(L3)] (Regularity) $\norm{A}_{C^{m,\alpha}(\mbb{R}^n)}\leq\sigma$,
\end{enumerate}
where $y,\xi\in\R^n$ and $k\in\Z^n$ and $\lambda$, $\Lambda$ and $\sigma$ are positive constants such that $\lambda\leq\Lambda$.

Our main result for linear equations can be summarized as follows.

\begin{thm}[{\bf Main Theorem I}]\label{thm:main-l} Let $m\geq 2$ be an integer and suppose that \eqref{eq:l} satisfies the structure conditions (L1)-(L3). Assume that $\{u^\e\}_{\e>0}$ is the family of the solutions of \eqref{eq:l} and $u$ is the homogenized limit of $\{u^\e\}_{\e>0}$ which solves \eqref{eq:l-h}. Then there are interior correctors $w_k^\e$ and boundary layer correctors $z_k^\e$, respectively defined by \eqref{eq:int cor-l-form-e} and \eqref{eq:bdry cor-l}, for $k=1,\ldots,m$ such that 
\begin{equation}\label{eq:main-l}
\norm{u^\e-\eta_m^\e-\theta_m^\e}_{L^\infty(\Omega)}\leq C\e^{m-1}
\end{equation}
for any $\e\in(0,1)$, where
\begin{equation*}
\eta_m^\e=u+\e w_1^\e+\e^2w_2^\e+\cdots+\e^mw_m^\e,\quad\theta_m^\e=\e z_1^\e+\e^2z_2^\e+\cdots+\e^mz_m^\e
\end{equation*}
on $\ov{\Omega}$ and $C$ is a positive constant depending only on $n,m,\alpha,\sigma,\lambda,\Lambda,\Omega,\norm{f}_{C^{m,\alpha}(\ov{\Omega})}$ and $\norm{g}_{C^{m+2,\alpha}(\ov{\Omega})}$. 
\end{thm}

\subsection{Fully nonlinear equations}\label{sec:nl-intro}
Set $\Omega$ to be a bounded domain of $\mbb{R}^n$ with $\p\Omega\in C^{m+2,1}$ and let $g\in C^{m+2,1}(\ov{\Omega})$. Suppose that $F\in C^m(\mc{S}^n\times\ov{\Omega}\times\mbb{R}^n)$ possesses the following properties.  

\begin{enumerate}
\item[(F1)] (Periodicity) $F(M,x,y+k)=F(M,x,y)$;
\item[(F2)] (Uniform Ellipticity) $\lambda\norm{N}\leq F(M+N,x,y)-F(M,x,y)\leq\Lambda\norm{N}$;
\item[(F3)] (Regularity) $\norm{F}_{C^{m,1}(\ov{B}_L\times\ov{\Omega}\times\R^n)}\leq\sigma(1+L)$;
\item[(F4)] (Concavity) $F(tM+(1-t)P)\geq tF(M)+(1-t)F(P)$,
\end{enumerate}
where $M$, $N$, $P\in\mc{S}^n$ with $N\geq 0$, $x\in\ov{\Omega}$, $y\in\R^n$, $k\in\Z^n$, and $t\in[0,1]$ and $L>0$, and $\lambda$, $\Lambda$ and $\sigma$ are positive constants such that $\lambda\leq\Lambda$. The concavity (F4) is assumed to obtain a $C^{2,\alpha}$ interior corrector and a concave effective operator $\ov{F}$, by which we get a sufficiently smooth limit profile.

Our main result for fully nonlinear equations is summarized in the following theorem.

\begin{thm}[{\bf Main Theorem II}]\label{thm:main-nl}
Let $m\geq 2$ and assume that $F$ satisfies the structure conditions (F1)-(F4), $g\in C^{m+2,1}(\ov{\Omega})$ and $\p\Omega\in C^{m+2,1}$. Then there are interior correctors $w_k^\e$ for $k=1,\ldots,[\frac{m}{2}]+1$ and the boundary layer corrector $\theta_m^\e$, respectively defined by \eqref{eq:int cor-nl-form-e} and \eqref{eq:bdry cor-nl} such that for any $\e_*\in(0,1)$, 
\begin{equation}\label{eq:main-nl}
\norm{u^\e-\eta_m^\e-\theta_m^\e}_{L^\infty(\Omega)}\leq C\e^{[\frac{m}{2}]},\quad\forall \e\in(0,\e_*],
\end{equation}
where 
\begin{equation*}
\eta_m^\e=u+\e w_1^\e+\e^2w_2^\e+\cdots+\e^{[\frac{m}{2}]+1}w_{[\frac{m}{2}]+1}^\e
\end{equation*}
on $\ov{\Omega}$ and $C>0$ depends only on $n,m,\e_*,\sigma,\lambda,\Lambda,F,g$ and $\Omega$.
\end{thm}

\subsection{Main steps}\label{sec:step-intro} 

In this subsection, we summarize the main strategies of this paper and make a few remarks on the key features observed in achieving the rates.
 
\paragraph{Higher order correctors and regularity theory}\label{para:1}
In order to find the next order approximation, we consider the linearized operator near the previous approximation. Since the linearized operator belongs to the same class of the previous one, we are able to proceed our argument in an inductive manner. The relationship between the current approximation and the next one is quite complicated in the nonlinear setting, unlike the linear case; however, such difficulty could be overcome by capturing the stability of correctors with respect to the shape of the limit profile, but not to the physical variable $x$. 

\paragraph{Induction arguments and compatibility conditions}\label{para:2}
Our induction argument consists of two substeps at each main step. First substep is to improve the previous approximation by constructing a globally periodic corrector and then bending it based on the shape of the limit profile. Then the improved interior approximation creates new errors, of a higher order, away from the given boundary data. The second substep is to fix the new errors by constructing a boundary layer corrector.
 
Additionally it is noteworthy that at each step of finding the $k$-th order interior corrector, we encounter a compatibility condition which uniquely determines the $(k-2)$-th order interior corrector. It illustrates the reason why the higher order asymptotic expansion \eqref{eq:exp} starts from $\e$-order but not from $\e^2$-order, as seen in many literatures (e.g., \cite{E1,E2}). It is closely related to the invariance of the quadratic rescaling of the governing equation.

\paragraph{Linearization and coupling effects}\label{para:3} 
There are two main differences between the linear and fully nonlinear settings. First the asymptotic expansion \eqref{eq:exp} is made inside of the operator for the fully nonlinear case, which creates an additional error unlike the linear case. Readers may compare the equation \eqref{eq:int cor-nl} to \eqref{eq:int cor-l}. Fortunately, all the additional errors are controllable and have no influence on determining the order of convergence rates.

Secondly, there is a coupling effect of the fast variable $y=\e^{-1}x$ and the slow variable $x$ of the interior correctors in the fully nonlinear case, unlike the linear case. Moreover, it causes the difference in the order of convergence rates as seen in Main Theorem I and II. The order is closely related to the regularity of interior correctors, and the coupling effect in the nonlinear case forces the next corrector to have two ``degrees'' less regularity than the current one (see Lemma \ref{lem:int cor-nl}).

\subsection{Historical background}\label{sec:bg-intro}
Classical results in the theory of homogenization could be found in the books \cite{BLP} and \cite{BP}, and the references therein. In particular, the notion of higher order correctors are introduced in these books, and one can find a higher order convergence rate for divergent operator on 1-dimensional space. This problem, however, is still open for higher dimensions where boundary oscillation plays a crucial role. 

Periodic homogenizations for first and second order nonlinear equations have been studied by many authors, such as Lions, Papanicolaou and Varadhan \cite{LPV}, Evans \cite{E1,E2}, Caffarelli \cite{C} and Majda and Souganidis \cite{MS} and Evans and Gomes \cite{EG}, etc. For homogenization with respect to an almost periodic or stationary ergodic environment has been considered by Ishii \cite{I}, Lions and Souganidis \cite{LS} and Caffarelli, Souganidis and Wang \cite{CSW}, etc.

Rates of convergence in the theory of periodic homogenization were considered by several authors in various circumstances; for example, Capuzzo Dolcetta and Ishii \cite{CI} and Camilli and Marchi \cite{CM} and Marchi \cite{M}, etc. In a stationary ergodic setting, also see Caffarelli and Souganidis \cite{CS}. However, as far as we know, there has been no literature concerning higher order convergence rates for homogenization of both linear and nonlinear elliptic equations in nondivergence form.

\subsection{Outline}\label{sec:ol-intro}
This paper is organized as follows. Section \ref{sec:l} is devoted to linear equations. We review the basic homogenization scheme via the viscosity method in Subsection \ref{sec:review-l}. Interior and boundary layer correctors of higher order are obtained in Subsection \ref{sec:int cor-l}. We present the proof of Main Theorem I in Subsection \ref{sec:pf main-l}. Section \ref{sec:nl} is devoted to fully nonlinear equations. The basic homogenization scheme of fully nonlinear equations is shown in Subsection \ref{sec:review-nl}. In Subsection \ref{sec:reg-nl} we investigate the regularity of the effective operator and the corrector function in the slow variable. In Subsection \ref{sec:int/bdry cor&rate-nl} we seek the higher order interior and boundary layer correctors, and finally prove Main Theorem II in Subsection \ref{sec:pf main-nl}. 

\subsection{Basic notations and terminologies}\label{sec:not-intro}
\begin{itemize}

\item $\mc{S}^n$ is the space of all $n\times n$ symmetric matrices. $\norm{M}$ denotes the ($L^2,L^2$)-norm of $M$ (i.e., $\norm{M}=\sup_{|x|=1}|Mx|$). 

\item $B_r(x)=\{y:\norm{y-x}<r\}$ where $x$ can be a point in $\R^n$ or $\mc{S}^n$. By $B_r$ we denote $B_r(0)$.

\item $Q_r(x)=(-r/2,r/2)^n\subset\mbb{R}^n$. As above, by $Q_r$ we denote $Q_r(0)$.

\item For the defintion of the classes $S(\lambda,\Lambda,f)$ and $S^*(\lambda,\Lambda,f)$ of viscosity solutions, see \cite{CC}.

\item Given $\vp,\psi\in C(\Omega)$, $\vp$ is said to touch $\psi$ by above (resp., by below) at $x_0$ in $\Omega$ if $\vp(x)\geq\psi(x)$ [resp., $\vp(x)\leq\psi(x)$] for all $x\in\Omega$ and $\vp(x_0)=\psi(x_0)$. 

\item $C^{k,\alpha}(\Omega)$ and $C^{k,\alpha}(\ov{\Omega})$ denote H\"{o}lder ($0<\alpha<1$) and Lipschitz ($\alpha=1$) spaces. $\norm{\cdot}_{C^{k,\alpha}(\Omega)}^*$ and $\norm{\cdot}_{C^{k,\alpha}(\Omega)}^{(j)}$ are adimensional norms (see Chapter 4 of \cite{GT}).

\item $D_pF$ [resp., $D_xF$] is the partial derivative in $M$-variable [resp., in $x$-variable].

\item We use the summation convention of repeated indices.

\item Unless otherwise stated, we always follow the following convention of constants: By $c_n, C_n$ we denote dimensional constants. By $c_0,c,C_0,C$ we denote the positive constants which depends only on the structure constants appearing in the structure conditions (L1)-(L3) or (F1)-(F4). By $C_{f_1,\ldots,f_k}$ and $C(f_1,\cdots,f_k)$ we denote positive constants depending on the constants in the structure conditions and further on $f_1,\ldots,f_k$ where $f_i$ can be a constant, a function, etc.
\end{itemize}


\section{Linear Equations in Non-divergence Form}\label{sec:l}

\subsection{Basic homogenization scheme}\label{sec:review-l}

Let us fix $\e>0$. The coefficient matrix $(a_{ij}(\cdot/\e))$ of \eqref{eq:l} is uniformly elliptic in $\ov{\Omega}$ with constants $\lambda$ and $\Lambda$, and belongs to $C^{m,\alpha}(\ov{\Omega})$. According to Theorem \ref{thm:reg-nl-apndx} (e) and (f), there exists a unique solution $u^\e\in C^{m+2,\alpha}(\ov{\Omega})$ of \eqref{eq:l}. In \cite{E2} it is shown that $\{u^\e\}_{\e>0}$ is uniformly bounded in $C^\alpha(\ov{\Omega})$ and hence has a limit. For the sake of completeness, we prove a weaker result that $\{u^\e\}_{\e>0}$ has a uniform modulus of continuity, which still guarantees the existence of limit.

\begin{lem}\label{lem:limit-l} Let $\{u^\e\}_{\e>0}\subset C^{m+2,\alpha}(\ov{\Omega})$ be the unique family which solve \eqref{eq:l} for each $\e>0$. Then there is a function $u\in C(\ov{\Omega})$ and a subsequence $\{u^{\e_k}\}_{k=1}^\infty$ of $\{u^\e\}_{\e>0}$ such that $u^{\e_k}\ra u$ uniformly in $\ov{\Omega}$ as $k\ra\infty$. 
\end{lem}

\begin{proof} 
We have $u^\e\in S(\lambda,\Lambda,f)$ in $\Omega$ for all $\e>0$ by the assumption (L2). By the setting, $g$ has a modulus of continuity $\rho(r)=[g]_{C^\alpha(\ov{\Omega})}r^\alpha$. Since $\p\Omega\in C^{m+2,\alpha}$, $\Omega$ satisfies a uniform sphere condition, say with radius $R>0$. Thus, Theorem \ref{thm:reg-nl-apndx} (d) implies that $u^\e$ has a modulus of continuity $\rho^*$, which depends only on $n,\lambda,\Lambda,\norm{f}_{L^\infty(\Omega)},\norm{g}_{L^\infty(\Omega)},\diam(\Omega),R$ and $\rho$. 

As the modulus of continuity $\rho^*$ is independent on $\e$, the family $\{u^\e\}_{\e>0}$ is equicontinuous on $\ov{\Omega}$. Moreover, by an a priori estimate we have $\norm{u^\e}_{L^\infty(\Omega)}\leq C(\norm{f}_{L^\infty(\Omega)}+\norm{g}_{L^\infty(\Omega)})$, where $C$ depends only on $n$, $\lambda$, $\Lambda$ and $\diam(\Omega)$, for each $\e>0$.

Now the conditions for the Arzela-Ascoli theorem are met, which ensures the existence of a subsequence $\{u^{\e_k}\}_{k=1}^\infty$ of $\{u^\e\}_{\e>0}$ which converges uniformly in $\ov{\Omega}$.
\end{proof}

The limit function $u$ will later turn out to be unique and satisfy \eqref{eq:l-h} in the classical sense. The next lemma plays a key role in proving this fact. The proof can be also found in \cite{AL}; nevertheless we contain the proof for completeness.

\begin{lem}\label{lem:key-l}
For each $M\in\mc{S}^n$ there exists a unique $\gamma\in\mbb{R}$ for which the following equation admits a 1-periodic solution 
\begin{equation}\label{eq:key-l}
a_{ij}D_{y_iy_j}w+a_{ij}M_{ij}=\gamma\quad\text{in }\mbb{R}^n.
\end{equation}
Moreover, the solutions of \eqref{eq:key-l} lie in $C^{2,\alpha}(\mbb{R}^n)$ and are unique up to an additive constant.
\end{lem}

To prove this lemma we consider the following penalized problem for $\delta\in(0,1)$

\begin{lem}\label{lem:key-l-d} Let $M\in\mc{S}^n$. There exists a unique bounded 1-periodic solution $w^\delta$ of 
\begin{equation}\label{eq:key-l-d}
a_{ij}D_{y_iy_j}w^\delta+a_{ij}M_{ij}-\delta w^\delta=0\quad\text{in }\mbb{R}^n.
\end{equation}
for each $\delta\in(0,1)$. Moreover, $w^\delta$ lies in $C^{2,\alpha}(\mbb{R}^n)$ with the estimate
\begin{equation}\label{eq:key-l-d-C2a-dwd}
\sup_{0<\delta<1}\norm{\delta w^\delta}_{C^{2,\alpha}(\mbb{R}^n)}\leq C\norm{M}.
\end{equation}
\end{lem}

\begin{proof}
In view of Theorem \ref{thm:exist-nl-apndx} (a) (with $F(N,y)=a_{ij}(y)N_{ij}+a_{ij}(y)M_{ij}$), we know that \eqref{eq:key-l-d} has a comparison principle. By the hypothesis (L2), all the eigenvalues of $(a_{ij})$ lie in the interval $[\lambda,\Lambda]$, which implies that
\begin{equation}\label{eq:key-l-d-Ca-AM}
\norm{a_{ij}M_{ij}}_{L^\infty(\mbb{R}^n)}\leq n\norm{A(y)}_{C^\alpha(\mbb{R}^n)}\norm{M}\leq n\sigma\norm{M}.
\end{equation}
It then follows that the constant functions $w_-^\delta=-n\sigma\norm{M}/\delta$ and $w_+^\delta=n\sigma\norm{M}/\delta$ are a subsolution and a supersolution respectively to \eqref{eq:key-l-d} for each $\delta\in(0,1)$. Thus, Perron's method (Theorem \ref{thm:exist-nl-apndx} (b) with $F(N,y)=a_{ij}(y)N_{ij}+a_{ij}(y)M_{ij}$, $u=w_-^\delta$ and $v=w_+^\delta$) ensures that there is a unique bounded 1-periodic viscosity solution $w^\delta\in C(\mbb{R}^n)$. It is immediate that
\begin{equation}\label{eq:key-l-d-linf-dwd}
\sup_{0<\delta<1}\norm{\delta w^\delta}_{L^\infty(\mbb{R}^n)}\leq n\sigma\norm{M}.
\end{equation}

Let us apply an interior Schauder estimate in a ball $B_{\sqrt{n}}(y_0)$ for $y_0\in\mbb{R}^n$ (see Theorem \ref{thm:reg-nl-apndx} (e)). Then $w^\delta\in C^{2,\alpha}(B_{\sqrt{n}/2}(y_0))$ and there is $c_0$ such that
\begin{equation*}
\norm{w^\delta}_{C^{2,\alpha}(B_{\sqrt{n}/2}(y_0))}^*\leq c_0\left(\norm{w^\delta}_{L^\infty(B_{\sqrt{n}}(y_0))}+n\sigma\norm{M}\right)\leq 2n\delta^{-1}c_0\sigma\norm{M}.
\end{equation*}
Since $y_0$ was chosen in an arbitrary way and $B_{\sqrt{n}/2}(y_0)$ contains a periodic cube, the estimate \eqref{eq:key-l-d-C2a-dwd} is verified with $C=2n\delta^{-1}c_0\sigma$.
\end{proof}

We observe that the oscillation of $w^\delta$ is bounded independent of $\delta$, although its $L^\infty$ norm is not bounded in a uniform way.

\begin{lem}\label{lem:key-l-osc} Let $M\in\mc{S}^n$ and $w^\delta$ be the unique solution to \eqref{eq:key-l-d}. Then 
\begin{equation}\label{eq:key-l-osc}
\sup_{0<\delta<1}\osc_{\mbb{R}^n}w^\delta\leq C\norm{M}.
\end{equation}
Moreover,
\begin{equation}\label{eq:key-l-d-C2a-twd}
\sup_{0<\delta<1}\norm{\tilde{w}^\delta}_{C^{2,\alpha}(B_1(y_0))}
\leq C\norm{M},
\end{equation}
where $\tilde{w}^\delta:=w^\delta-w^\delta(0)$.
\end{lem}

\begin{proof} Define $\hat{w}^\delta(y):=w^\delta(y)-\min_{\mbb{R}^n}w^\delta\geq 0$ in $\R^n$. Note that $\hat{w}^\delta$ and $w^\delta$ achieve its global maximum and minimum, and $\hat{w}^\delta\in C^{2,\alpha}(\R^n)$. Additionally, $\osc_{\mbb{R}^n}w^\delta=\max_{\mbb{R}^n}\hat{w}^\delta$. Moreover, plugging $\hat{w}^\delta$ into \eqref{eq:key-l-d} we obtain
\begin{equation}\label{eq:key-l-d-whd}
a_{ij}D_{y_iy_j}\hat{w}^\delta-\delta\hat{w}^\delta=\delta\min_{\mbb{R}^n}w^\delta-a_{ij}M_{ij}\quad\text{in }\mbb{R}^n.
\end{equation}

Let us restrict our domain to $B_{\sqrt{n}}(y_0)$ where $y_0$ is an arbitrary point in $\mbb{R}^n$. Note that $B_{\sqrt{n}/2}(y_0)$ contains a periodic cube $Q_1(y_0)$. This implies that $\sup_{B_{\sqrt{n}/2}(y_0)}\hat{w}^\delta=\sup_{\mbb{R}^n}\hat{w}^\delta$ and $\inf_{B_{\sqrt{n}/2}(y_0)}\hat{w}^\delta=\inf_{\mbb{R}^n}\hat{w}^\delta=0$. Now we apply the Harnack inequality over $B_{\sqrt{n}}(y_0)$ to \eqref{eq:key-l-d-whd} (see Theorem \ref{thm:reg-nl-apndx} (a) with $f=\delta\min_{\mbb{R}^n}w^\delta-a_{ij}M_{ij}$). Then 
\begin{equation*}
\sup_{B_{\sqrt{n}/2}(y_0)}\hat{w}^\delta\leq c_0\norm{\lambda^{-1}(\delta\min_{\mbb{R}^n}w^\delta-a_{ij}M_{ij})}_{L^\infty(B_{\sqrt{n}}(y_0))}\leq 2c_0\lambda^{-1}n\sigma\norm{M};
\end{equation*}
here we utilized \eqref{eq:key-l-d-Ca-AM} and \eqref{eq:key-l-d-linf-dwd}. Since the above bound is independent of $\delta\in(0,1)$, and since $y_0$ is an arbitrary point, we have shown \eqref{eq:key-l-osc} with $C=2c_0\lambda^{-1}n\sigma$.

Define now $\tilde{w}^\delta(y):=w^\delta(y)-w^\delta(0)$ in $\R^n$. By \eqref{eq:key-l-osc}, $|\tilde{w}^\delta|\leq\tilde{c}_0\norm{M}$ in $\R^n$ where $\tilde{c}_0=4c_0\lambda^{-1}n\sigma$. Moreover, $\tilde{w}^\delta\in C^{2,\alpha}(\mbb{R}^n)$ and satisfies
\begin{equation*}
a_{ij}D_{y_iy_j}\tilde{w}^\delta+a_{ij}M_{ij}-\delta\tilde{w}^\delta=\delta w^\delta(0)\quad\text{in }\mbb{R}^n.
\end{equation*} 
Using a similar argument when proving \eqref{eq:key-l-d-C2a-dwd}, we get 
\begin{equation*}
\begin{split}
\sup_{0<\delta<1}\norm{\tilde{w}^\delta}_{C^{2,\alpha}(B_1(y_0))}&\leq\tilde{c}_1c_0n\sigma(\lambda^{-1}+1)\norm{M},
\end{split}
\end{equation*}
which verifies \eqref{eq:key-l-d-C2a-twd} with $C=\tilde{c}_1c_0n\sigma(\lambda^{-1}+1)$.
\end{proof}

Now we are ready to prove Lemma \ref{lem:key-l}

\begin{proof}[Proof of Lemma \ref{lem:key-l}] In view of \eqref{eq:key-l-d-linf-dwd}, we can take a subsequence $\{\delta_k w^{\delta_k}(0)\}_{k=1}^\infty$ of $\{\delta w^\delta\}_{0<\delta<1}$ and a number $\gamma\in\mbb{R}$ such that $\delta_k w^{\delta_k}(0)\ra\gamma$ as $k\ra\infty$. Then \eqref{eq:key-l-osc} implies that $\delta_k w^{\delta_k}\ra\gamma$ uniformly in $\mbb{R}^n$ as $k\ra\infty$. 

On the other hand, by the compact embedding, the uniform estimate \eqref{eq:key-l-d-C2a-twd} yields that 
\begin{equation}\label{eq:key-l-conv}
\norm{\delta_kw^{\delta_k}-\gamma}_{L^\infty(\mbb{R}^n)}+\norm{\tilde{w}^{\delta_k}-w}_{C^2(\mbb{R}^n)}\lra 0\quad\text{as}\quad k\lra\infty,
\end{equation}
for some 1-periodic $w\in C^{2,\alpha}(\R^n)$. Note that one may need to take a further subsequence of $\{\delta_k\}_{k=1}^\infty$ to ensure the convergence above.  

By the stability of viscosity solutions, $w$ solves \eqref{eq:key-l} in the viscosity sense. Then the $C^{2,\alpha}(\mbb{R}^n)$-regularity of $w$ forces itself to be a classical solution. 

To this end we prove that the constant $\gamma$ is unique. Suppose to the contrary that there is another $\gamma'\in\mbb{R}$ to which a subsequence of $\{\delta w^\delta\}_{0<\delta<1}$ converges uniformly in $\mbb{R}^n$. Denote $w'$, which belongs to $C^{2,\alpha}(\mbb{R}^n)$, by the corresponding limit of a subsequence of $\{\tilde{w}^\delta\}_{0<\delta<1}$. 

Assume without lose of generality that $\gamma<\gamma'$. As $w$ and $w'$ being bounded, we are able to add a constant $t_0$ to $w$ in such a way that $w'(y_0)+t_0<w(y_0)$ at a point $y_0\in\mbb{R}^n$. Take $t_1$ by the infimum value of $t$ such that $w'+t\geq w$ in $\mbb{R}^n$. Then $w'+t_1$ touches $w$ by above at a point $y_1$. Since $w$ is a solution of \eqref{eq:key-l},
\begin{equation*}
\gamma\leq a_{ij}(y_1)D_{y_iy_j}(w'+t_1)(y_1)+a_{ij}(y_1)M_{ij}=\gamma',
\end{equation*}
which is a contradiction. It shows that the constant $\gamma$ must be unique.

Furthermore, the Liouville theorem (e.g., Theorem \ref{thm:reg-nl-apndx}) implies that the uniform convergence \eqref{eq:key-l-conv} could be made along the full sequence; i.e., the limit function is also unique.

The last assertion of Lemma \ref{lem:key-l} is also an easy consequence of the Liouville theorem.
\end{proof}

From now on we denote $w^\delta(\cdot;M)$ by the unique solution of \eqref{eq:key-l-d} for a given $M\in\mc{S}^n$. Also $\hat{w}^\delta(\cdot;M):=w^\delta(\cdot;M)-\min_{\mbb{R}^n}w^\delta(\cdot;M)$ and $\tilde{w}^\delta(\cdot;M):=w^\delta(\cdot;M)-w^\delta(0;M)$. In addition, let us write $w(\cdot;M)$ by the solution of \eqref{eq:key-l} for a given $M\in\mc{S}^n$ which is normalized by $0$; i.e., $w(0;M)=0$. 

By Lemma \ref{lem:key-l} we can understand $\gamma$ as a functional $M\mapsto\gamma(M)$ on $\mc{S}^n$. The linear structure of the equation \eqref{eq:key-l} allows us to obtain further information about the functional $\gamma$ which is stated in the next lemma. 

\begin{lem}\label{lem:eff op-l}
Let $\gamma$ be the functional on $\mc{S}^n$ obtained from Lemma \ref{lem:key-l}. 
\begin{enumerate}[(i)]
\item There is a constant symmetric matrix $(\ov{a}_{ij})$ such that $\gamma(M)=\ov{a}_{ij}M_{ij}$. 
\item The matrix $(\ov{a}_{ij})$ is elliptic with the same ellipticity constants of $(a_{ij})$; i.e., $\lambda|\xi|^2\leq\ov{a}_{ij}\xi_i\xi_j\leq\Lambda|\xi|^2$ for all $\xi\in\R^n$.
\end{enumerate}
\end{lem}

\begin{proof} The assertion (i) is a direct consequence of Lemma \ref{lem:key-l-d}, and is left to the readers. 

We prove the assertion (ii). Since the proofs are similar, we only show the first inequality. Choose any $\e>0$ and assume for a contradiction that there exists $\xi\in\mbb{R}^n$ for which $\ov{a}_{ij}\xi_i\xi_j<(\lambda-\e)|\xi|^2$. 
In view of \eqref{eq:key-l-conv}, there corresponds $\delta\in(0,1)$ for which $\norm{\delta w^\delta(\cdot;\xi\cdot\xi^t)-\ov{a}_{ij}\xi_i\xi_j}_{L^\infty(\mbb{R}^n)}<\e|\xi|^2$. For the moment we abbreviate $w^\delta(\cdot;\xi\cdot\xi^t)$ by $w^\delta$. Then
\begin{equation*}
a_{ij}D_{y_iy_j}w^\delta=\delta w^\delta-a_{ij}\xi_i\xi_j\leq\norm{\delta w^\delta-\ov{a}_{ij}\xi_i\xi_j}_{L^\infty(\mbb{R}^n)}+(\ov{a}_{ij}\xi_i\xi_j-\lambda|\xi|^2)<0
\end{equation*}
in $\R^n$, which is contradictory to the fact that $w^\delta$ achieves a global minimum.
\end{proof}

The constant matrix $(\ov{a}_{ij})$ from Lemma \ref{lem:eff op-l} is called {\it the effective coefficients} of $(a_{ij})$ in the following lemma. It is proved in \cite{E2}, but we present the proof for completeness. 

\begin{lem}\label{lem:main-l-h} Suppose that \eqref{eq:l} satisfies the structure conditions (L1)-(L2) and let $\{u^\e\}_{\e>0}\subset C^{m+2,\alpha}(\Omega)$ be the family of solutions to \eqref{eq:l}. Then there exists a unique function $u$, which has a modulus of continuity on $\ov{\Omega}$, such that $u^\e\ra u$ uniformly in $\ov{\Omega}$ as $\e\ra 0$. Moreover, $u\in C^{m+2,\alpha}(\ov{\Omega})$ and it solves 
\begin{equation}\tag{$\ov{L}$}\label{eq:l-h}
\begin{cases}
\ov{a}_{ij}D_{ij}u=f&\text{in }\Omega,\\
u=g&\text{on }\p\Omega,
\end{cases}
\end{equation}
\end{lem}

\begin{proof} We already proved part of the first assertion in Lemma \ref{lem:limit-l}. Since $u^\e\ra u$ uniformly in $\ov{\Omega}$ up to a subsequence and $u^\e=g$ on $\p\Omega$ for all $\e>0$, we have $u=g$ on $\p\Omega$. On the other hand, the maximum principle implies that \eqref{eq:l-h} has at most one solution. Therefore, the convergence of $u^\e\ra u$ is valid without extracting a subsequence.

We claim that $u$ is a viscosity solution to \eqref{eq:l-h}. If it is true, then Theorem \ref{thm:reg-nl-apndx} (e) and (f) imply that $u\in C^{m+2,\alpha}(\ov{\Omega})$. 

Thus, we are only left with proving the above claim. Let $P$ be a paraboloid which touches $u$ by above at $x_0$ in a neighborhood. By replacing $P$ by $P+\eta|x-x_0|^2$ ($\eta>0$) we may assume that $P$ touches $u$ strictly by above. Assume to the contrary that $\ov{a}_{ij}D_{ij}P-f(x_0)<0$. By the continuity of $f$, we can choose $r>0$ in such a way that $B_r(x_0)\subset\Omega$ and $\ov{a}_{ij}D_{ij}P-f(x)<0$ for any $x\in B_r(x_0)$. 

Define $P^\e(x):=P(x)+\e^2w(\e^{-1}x;D^2P)$. Note that $P^\e\in C^{2,\alpha}(\ov{\Omega})$. In view of \eqref{eq:key-l} we obtain
\begin{equation*}
a_{ij}\left(\frac{x}{\e}\right)D_{ij}P^\e(x)-f(x)=\ov{a}_{ij}D_{ij}P-f(x)<0\quad\text{in }B_r(x_0).
\end{equation*}
Hence, $P^\e$ is a supersolution of \eqref{eq:l} so that the strong maximum principle implies $(u^\e-P^\e)(x_0)<\max_{\p B_r(x_0)}(u^\e-P^\e)$. Letting $\e\ra 0$ then gives $\max_{\p B_r(x_0)}(u-P)\geq 0$, which violates the assumption that $P$ touches $u$ strictly by above at $x_0$. Therefore, $\ov{a}_{ij}D_{ij}P-f(x)\geq 0$ for any $x\in\Omega$. It shows that $u$ is a viscosity subsolution of \eqref{eq:l-h}. 

In a similar manner, we are able to prove that $u$ is a viscosity supersolution of \eqref{eq:l-h}. This completes the proof. 
\end{proof}

\subsection{Interior and boundary layer correctors}\label{sec:int cor-l}

In this subsection, we seek the interior and boundary layer correctors. We make a remark from the previous section before we begin. Recall from the linear algebra, $\{E^{ij}|i,j=1,\ldots,n\}$ is the standard basis of $\mc{S}^n$. Any matrix $M\in\mc{S}^n$ can be written as $M=M_{ij}E^{ij}$ where $M=(M_{ij})$. Set $M=E^{kl}$ in Lemma \ref{lem:key-l} for $k,l\in\{1,\ldots,n\}$ and write $\chi^{kl}:=w(\cdot;E^{kl})\in C^{2,\alpha}(\mbb{R}^n)$. Notice that $\chi^{kl}(0)=0$. In view of \eqref{eq:key-l} and Lemma \ref{lem:eff op-l} (i), $\chi^{kl}$ solves 
\begin{equation}\label{eq:chi-l-2}
a_{ij}D_{ij}\chi^{kl}+a_{kl}=\ov{a}_{kl}.
\end{equation}
Multiplying \eqref{eq:chi-l-2} with $M_{kl}$ and summing over the indices $k,l=1,\ldots,n$, we see that $\chi^{kl}M_{kl}$ solves \eqref{eq:key-l} with $M=(M_{kl})$. Define
\begin{equation*}
w_2(y,x)=\chi^{kl}(y)D_{x_kx_l}u(x)+\psi_2(x)\quad(y\in\mbb{R}^n,x\in\ov{\Omega}),
\end{equation*}
where $u$ is given by Lemma \ref{lem:main-l-h} and $\psi_2$ is chosen arbitrarily from $C^{m,\alpha}(\ov{\Omega})$ for the moment. By Lemma \ref{lem:main-l-h}, $w_2(\cdot,x)\in C^{2,\alpha}(\mbb{R}^n)$ for each $x\in\ov{\Omega}$ while $w_2(y,\cdot)\in C^{m,\alpha}(\ov{\Omega})$ for each $y\in\mbb{R}^n$. Moreover, $w_2(\cdot,x)$ solves 
\begin{equation*}
a_{ij}D_{y_iy_j}w_2(\cdot,x)+a_{ij}D_{x_ix_j}u(x)=0\quad\text{in }\mbb{R}^n
\end{equation*}
for each $x\in\ov{\Omega}$. We call $w_2$ the second order (interior) corrector of \eqref{eq:l}. The first order corrector will be defined afterward as a compatibility condition of the third order corrector. 

Interior correctors of higher orders are discovered in the similar direction.

\begin{lem}\label{lem:chi-l} There are a family $\{\ov{a}_{i_1\ldots i_k}|1\leq i_1,\ldots,i_k\leq n,k\geq 2\}$ of constants and a family $\{\chi^{i_1\ldots i_k}|1\leq i_1,\ldots,i_k\leq n,k\geq 2\}$ of 1-periodic functions in $C^{2,\alpha}(\mbb{R}^n)$ which satisfy the following recursive equation
\begin{equation}\label{eq:chi-l}
a_{ij}D_{ij}\chi^{i_1\ldots i_k}+2a_{i_kj}D_{y_j}\chi^{i_1\ldots i_{k-1}}+a_{i_{k-1}i_k}\chi^{i_1\ldots i_{k-2}}=\ov{a}_{i_1\ldots i_k}\quad\text{in }\mbb{R}^n
\end{equation}
for each $1\leq i_1,\ldots,i_k\leq n$. Here we understand $\chi\equiv 1$ and $\chi^{i}\equiv 0$ for each $i=1,\ldots,n$. Furthermore, for each $k\geq 2$, $\chi^{i_1\ldots i_k}(0)=0$ and 
\begin{equation}\label{eq:chi-l-C2a-ova/chi}
|\ov{a}_{i_1\ldots i_k}|+\norm{\chi^{i_1\ldots i_k}}_{C^{2,\alpha}(\mbb{R}^n)}\leq C_k,\quad\forall 1\leq i_1,\ldots,i_k\leq n.
\end{equation}
\end{lem}

\begin{proof} We already know $\{\ov{a}_{ij}\}_{i,j=1,\ldots,n}$ and $\{\chi^{ij}\}_{i,j=1,\ldots,n}$ from the comment above this lemma; one may notice that \eqref{eq:chi-l} is exactly the same with \eqref{eq:chi-l-2} if $k=2$. The constant $C_2$ can be taken by the sum of those from \eqref{eq:key-l-d-C2a-dwd} and \eqref{eq:key-l-d-C2a-twd}.

The construction of the families $\{\ov{a}_{i_1\ldots i_k}\}$ and $\{\chi^{i_1\ldots i_k}\}$ (for $k\geq 3$) can be done by an induction argument, mainly following the lines of the proofs of Lemma \ref{lem:key-l}, \ref{lem:key-l-d} and \ref{lem:key-l-osc}. To avoid the redundancy, we leave it to the readers.
\end{proof}

Now let $m\geq 3$. By Lemma \ref{lem:main-l-h} we have $u\in C^{m+2,\alpha}(\ov{\Omega})$. For $1\leq k\leq m-2$, define $\psi_k\in C^{m-k+2,\alpha}(\ov{\Omega})$ recursively by the unique solution of 
\begin{equation}\label{eq:psi-l}
\begin{cases}
\ov{a}_{ij}D_{x_ix_j}\psi_k=-\sum_{l=3}^{k+2}\ov{a}_{i_1\ldots i_l}D_{x_{i_1}\ldots x_{i_l}}\psi_{k-l+2}&\text{in }\Omega,\\
\psi_k=0&\text{on }\p\Omega,
\end{cases}
\end{equation}
where we understand $\psi_0\equiv u$. This can be done by an induction argument. Fix $k$ and suppose that $\psi_l\in C^{m-l+2,\alpha}(\ov{\Omega})$ for all $0\leq l<k$. Then the right hand side of \eqref{eq:psi-l} belongs to $C^{m-k,\alpha}(\ov{\Omega})$. Now the existence and regularity theories ensure that the boundary value problem \eqref{eq:psi-l} attains a unique solution $\psi_k\in C^{m-k+2,\alpha}(\ov{\Omega})$. This induction holds because the induction hypothesis is met for $k=1$. 

Furthermore, we have the following.

\begin{lem}\label{lem:psi-l-est} Let $m\geq 3$ and set $\psi_k$ as above for $1\leq k\leq m-2$. Then 
\begin{equation}\label{eq:psi-l-est}
\norm{\psi_k}_{C^{m-k+2,\alpha}(\ov{\Omega})}\leq \tilde{C}_{k,m,\Omega}\left(\norm{f}_{C^{m,\alpha}(\ov{\Omega})}+\norm{g}_{C^{m+2,\alpha}(\ov{\Omega})}\right),
\end{equation}
for each $k=0,1,\ldots,m-2$, where we understand $\psi_0=u$. 
\end{lem}

\begin{proof} Since $u\in C^{m+2,\alpha}(\ov{\Omega})$ solves \eqref{eq:l-h} and since $f\in C^{m,\alpha}(\ov{\Omega})$, $g\in C^{m+2,\alpha}(\ov{\Omega})$ and $\p\Omega\in C^{m+2,\alpha}$, Theorem \ref{thm:reg-nl-apndx} (f) and an a priori estimate yield that
\begin{equation*}
\norm{u}_{C^{m+2,\alpha}(\ov{\Omega})}\leq C_{m,\Omega}\left(\norm{f}_{C^{m,\alpha}(\ov{\Omega})}+\norm{g}_{C^{m+2,\alpha}(\ov{\Omega})}\right).
\end{equation*}
The proof is finished by adopting an induction argument. One can also prove that 
\begin{equation*}
\tilde{C}_{k,m,\Omega}\leq C_{m-k+2,\Omega}\sum_{l=3}^{k+2}C_l\tilde{C}_{k-l+2,m,\Omega}.
\end{equation*}
\end{proof}

Set for each $1\leq k\leq m$
\begin{equation}\label{eq:int cor-l-form}
w_k(y,x)=\sum_{l=1}^{k}\chi^{i_1\ldots i_l}(y)D_{x_{i_1}\ldots x_{i_l}}\psi_{k-l}(x)+\psi_k(x)\quad(y\in\mbb{R}^n,x\in\ov{\Omega}),
\end{equation}
where $\psi_{m-1}\in C^{3,\alpha}(\ov{\Omega})$ and $\psi_m\in C^{2,\alpha}(\ov{\Omega})$ are arbitrary functions which satisfy the inequality \eqref{eq:psi-l-est} respectively when $k=m-1$ and $m$. Recall that we have set $\chi^i\equiv 0$ for all $i=1,\ldots,n$, which implies that $w_1(y,x)=\psi_1(x)$; that is, $w_1$ is independent of the $y$-variable. 
\begin{lem}\label{lem:int cor-l} Let $m\geq 3$ be an integer and $w_k$ be given by \eqref{eq:int cor-l-form} for each $k=1,\ldots,m$. Then $w_k(\cdot,x)\in C^{2,\alpha}(\mbb{R}^n)$ for each $x\in\ov{\Omega}$ and $w_k(y,\cdot)\in C^{m-k+2,\alpha}(\ov{\Omega})$ for each $y\in\mbb{R}^n$ with the estimate 
\begin{equation*}
\norm{w_k(\cdot,x)}_{C^{2,\alpha}(\mbb{R}^n)}+\norm{w_k(y,\cdot)}_{C^{m-k+2,\alpha}(\ov{\Omega})}\leq\bar{C}_{k,m,\Omega}\left(\norm{f}_{C^{m,\alpha}(\ov{\Omega})}+\norm{g}_{C^{m+2,\alpha}(\ov{\Omega})}\right),
\end{equation*}
where $\bar{C}_{k,m,\Omega}=\sum_{l=1}^kn^lC_l\tilde{C}_{k-l,m,\Omega}+\tilde{C}_{k,m,\Omega}$ for each $k=1,\ldots,m$. 

Moreover, for $3\leq k\leq m$, $w_k$ solves recursively
\begin{equation}\label{eq:int cor-l}
a_{ij}D_{y_iy_j}w_k+2a_{ij}D_{x_iy_j}w_{k-1}+a_{ij}D_{x_ix_j}w_{k-2}=0\quad\text{in }\mbb{R}^n\times\Omega.
\end{equation}
\end{lem}

\begin{proof} The estimate follows from \eqref{eq:chi-l-C2a-ova/chi} and \eqref{eq:psi-l-est}. The equation \eqref{eq:int cor-l} is immediate from \eqref{eq:chi-l} and \eqref{eq:psi-l}. 
\end{proof}

Define now the $k$-th order interior corrector $w_k^\e$ of \eqref{eq:l} for each $1\leq k\leq m$ and $\e>0$ by
\begin{equation}\label{eq:int cor-l-form-e}
w_k^\e(x):=w_k\left(\frac{x}{\e},x\right)\quad(x\in\ov{\Omega}).
\end{equation}
By Lemma \ref{lem:int cor-l}, $w_k^\e\in C^{2,\alpha}(\ov{\Omega})$ for each $\e>0$. Thus, the following boundary value problem has a unique solution lying in $C^{2,\alpha}(\ov{\Omega})$; 
\begin{equation}\label{eq:bdry cor-l}
\begin{cases}
a_{ij}\left(\frac{x}{\e}\right)D_{ij}z_k^\e=0&\text{in }\Omega,\\
z_k^\e=-w_k^\e&\text{on }\p\Omega.
\end{cases}
\end{equation}
We denote the solution by $z_k^\e$ and call it the $k$-th order boundary layer corrector of \eqref{eq:l}. Lemma \ref{lem:int cor-l} yields a uniform bound of $z_k^\e$, namely,
\begin{equation*}
\sup_{\e>0}\norm{z_k^\e}_{L^\infty(\Omega)}\leq c_0\sup_{\e>0}\norm{w_k^\e}_{L^\infty(\Omega)}\leq c_0\bar{C}_{k,m,\Omega}\left(\norm{f}_{C^{m,\alpha}(\ov{\Omega})}+\norm{g}_{C^{m+2,\alpha}(\ov{\Omega})}\right).
\end{equation*}
Note that for any $\e>0$, $z_1^\e\equiv 0$ on $\ov{\Omega}$, since $w_1^\e\equiv\psi_1$ on $\ov{\Omega}$ where $\psi_1$ vanishes on $\p\Omega$. 

\subsection{Proof of Main Theorem I}\label{sec:pf main-l}

We are now in position to prove Main Theorem I. 
\begin{proof}[Proof of Theorem \ref{thm:main-l}] Fix $\e>0$. Let $w_k^\e$ and $z_k^\e$ be defined as in the previous section for each $k=1,\ldots,m$. Define
\begin{equation*}
\eta_m^\e:=u+\e w_1^\e+\e^2w_2^\e+\cdots+\e^mw_m^\e,\quad\theta_m^\e:=\e z_1^\e+\e^2z_2^\e+\cdots+\e^mz_m^\e
\end{equation*}
on $\ov{\Omega}$. Then both $\eta_m^\e$ and $\theta_m^\e$ belong to $C^{2,\alpha}(\ov{\Omega})$. We utilize \eqref{eq:key-l}, \eqref{eq:int cor-l} and \eqref{eq:bdry cor-l}. A lengthy but elementary computation gives 
\begin{equation*}
\begin{split}
a_{ij}\left(\frac{x}{\e}\right)D_{ij}(\eta_m^\e+\theta_m^\e)=a_{ij}\left(\frac{x}{\e}\right)D_{ij}\eta_m^\e=f+\e^{m-1}\vp_m^\e
\end{split}
\end{equation*}
in $\Omega$, where
\begin{equation*}
\begin{split}
\vp_m^\e(x)&=\sum_{l=2}^{m-1}\left[2a_{i_lj}\left(\frac{x}{\e}\right)D_{y_j}\chi^{i_1\ldots i_{l-1}}\left(\frac{x}{\e}\right)+a_{i_{l-1}i_l}\left(\frac{x}{\e}\right)\chi^{i_1\ldots i_{l-2}}\left(\frac{x}{\e}\right)\right]\\
&\quad\quad\quad\quad\times D_{x_{i_1}\cdots x_{i_l}}\psi_{m-l-1}(x)\\
&\quad+\e\sum_{l=2}^ma_{i_{l-1}i_l}\left(\frac{x}{\e}\right)\chi^{i_1\ldots i_{l-2}}\left(\frac{x}{\e}\right)D_{x_{i_1}\ldots x_{i_l}}\psi_{m-l}(x)\qquad(x\in\Omega).
\end{split}
\end{equation*}
Now we set $\e\in(0,1)$. According to \eqref{eq:chi-l-C2a-ova/chi} and \eqref{eq:psi-l-est}, we have
\begin{equation*}
\norm{\vp_m^\e}_{L^\infty(\Omega)}\leq L_{m,\Omega}\left(\norm{f}_{C^{m,\alpha}(\ov{\Omega})}+\norm{g}_{C^{m+2,\alpha}(\ov{\Omega})}\right)
\end{equation*}
where 
\begin{equation*}
L_{m,\Omega}=\sigma\left[\sum_{l=3}^{m-1}n^{l-1}\left\{2(C_{l-1}+C_{l-2})\tilde{C}_{m-l-1,\Omega}+C_{l-2}\tilde{C}_{m-l,\Omega}\right\}+n^{m-1}C_{m-2}\tilde{C}_{0,\Omega}\right].
\end{equation*}
Here $C_k$ and $\tilde{C}_k$ are the constants chosen as in \eqref{eq:chi-l-C2a-ova/chi} and \eqref{eq:psi-l-est}. 

On the other hand, we have $\eta_m^\e+\theta_m^\e=g+\sum_{k=1}^m\e^k(w_k^\e+z_k^\e)=g$ on $\p\Omega$. Thus, $u^\e-\eta_m^\e-\theta_m^\e\in C^{2,\alpha}(\ov{\Omega})$ solves the following equation,
\begin{equation*}
\begin{cases}
a_{ij}\left(\frac{x}{\e}\right)D_{ij}v=-\e^{m-1}\vp_m^\e&\text{in }\Omega,\\
v=0&\text{on }\p\Omega.
\end{cases}
\end{equation*}
An a priori estimate then gives
\begin{equation*}
\norm{u^\e-\eta_m^\e-\theta_m^\e}_{L^\infty(\Omega)}\leq c_0L_{m,\Omega}\left(\norm{f}_{C^{m,\alpha}(\ov{\Omega})}+\norm{g}_{C^{m+2,\alpha}(\ov{\Omega})}\right).
\end{equation*}
\end{proof}

\section{Fully Nonlinear Equations in Non-divergence Form}\label{sec:nl}

\subsection{Basic homogenization scheme}\label{sec:review-nl}

This subsection is devoted to the homogenization process of \eqref{eq:nl} to \eqref{eq:nl-h}. It generalizes the homogenization result of linear equations (see Section \ref{sec:review-l}). One may find a general argument in \cite{E2} for some lemmas. However, we present all the proofs which are adequate for our situation. 

\begin{lem}\label{lem:limit-nl} Assume for each $\e>0$ that $u^\e\in C(\ov{\Omega})$ is a viscosity solution of \eqref{eq:nl}. Then there is a function $u\in C(\ov{\Omega})$ and a subsequence $\{u^{\e_k}\}_{k=1}^\infty$ of $\{u^\e\}_{\e>0}$ such that $u^{\e_k}\ra u$ uniformly in $\ov{\Omega}$ as $k\ra\infty$. 
\end{lem}

\begin{proof} The proof is identical to that of Lemma \ref{lem:limit-l}. One may notice that the proof of Lemma \ref{lem:limit-l} does not involve the linear structure of \eqref{eq:l}.
\end{proof}

As we did in Section \ref{sec:review-l}, we will ascertain the effective equation which $u$ solves in the viscosity sense at the end of this section. Before we start, we point out that the argument throughout this subsection is valid by only assuming that $F\in C^{0,1}(\ov{B}_L\times\ov{\Omega}\times\mbb{R}^n)$ for each $L>0$ (i.e., (F3) with $m=0$). 

\begin{lem}\label{lem:key-nl} To each $(M,x)\in\mc{S}^n\times\ov{\Omega}$ there corresponds a unique $\gamma\in\R$ for which the following equation
\begin{equation}\label{eq:key-nl}
F(D_y^2w+M,x,y)=\gamma\quad\text{in }\mbb{R}^n
\end{equation}
attains a 1-periodic solution $w\in C^{2,\alpha}(\mbb{R}^n)$. Moreover, $w$ is unique up to an additive constant. Moreover, if the solution $w$ satisfies $w(0)=0$, then 
\begin{equation}\label{eq:key-nl-est-w}
\norm{w}_{C^{2,\alpha}(\mbb{R}^n)}\leq C_{\norm{M}}.
\end{equation}
\end{lem}

As we did in the linear case, we start with an approximating problem.
 
\begin{lem}\label{lem:key-nl-d}
Let $(M,x)\in\mc{S}^n\times\ov{\Omega}$ and $\delta\in(0,1)$. Then there is a unique bounded 1-periodic function $w^\delta\in C^{2,\alpha}(\mbb{R}^n)$ which solves
\begin{equation}\label{eq:key-nl-d}
F(D_y^2w^\delta+M,x,y)-\delta w^\delta=0\quad\text{in }\mbb{R}^n,
\end{equation}
with the uniform estimate
\begin{equation}\label{eq:key-nl-d-C2a-dwd}
\sup_{0<\delta<1}\norm{\delta w^\delta}_{C^{2,\alpha}(\mbb{R}^n)}\leq C_{\norm{M}}.
\end{equation}
\end{lem}

\begin{proof} Fix $(M,x)\in\mc{S}^n\times\ov{\Omega}$. The unique existence of the solution $w^\delta$ to \eqref{eq:key-nl-d} follows the same argument as in Lemma \ref{lem:key-l-d}, so is omitted. Moreover, we have 
\begin{equation}\label{eq:key-nl-d-linf-dwd}
\sup_{0<\delta<1}\norm{\delta w^\delta}_{L^\infty(\mbb{R}^n)}\leq\sigma(1+\norm{M}).
\end{equation}

To improve the regularity of $w^\delta$ to $C^{2,\alpha}(\R^n)$ we make use of interior $C^{2,\alpha}$ estimate (Theorem \ref{thm:reg-nl-apndx} (e)) instead of the interior Schauder estimate. We know from the hypothesis (F4) that $F$ is concave with respect to $M$ and from the hypothesis (F3) that for any $y,y_0\in\R^n$
\begin{equation}\label{eq:key-nl-Ca-beta}
\beta(y,y_0):=\sup_{N\in\mc{S}^n}\frac{|F(M+N,x,y)-F(M+N,x,y_0)|}{1+\norm{N}}\leq \sigma(1+\norm{M})|y-y_0|.
\end{equation}

On the other hand, since $w^\delta$ is a solution to \eqref{eq:key-nl-d} in $\R^n$, we have $w^\delta\in S(\lambda/n,\Lambda,\delta w^\delta-F(M,x,\cdot))$ in $\mbb{R}^n$. As we restrict ourselves to the cube $Q_2$, we obtain from Theorem \ref{thm:reg-nl-apndx} (b) that $w^\delta\in C^{\tilde{\alpha}}(\ov{Q}_1)$ and $\norm{w^\delta}_{C^{\tilde{\alpha}}(\ov{Q}_1)}\leq c_0(\delta^{-1}+2)\sigma(1+\norm{M})$, for each $\delta>0$. Since $Q_1$ is a periodic cube of $w^\delta$, we obtain a uniform H\"{o}lder estimate on $\delta w^\delta$ over $\mbb{R}^n$, namely,
\begin{equation}\label{eq:key-nl-d-Ca-dwd}
\sup_{0<\delta<1}\norm{\delta w^\delta}_{C^{\bar{\alpha}}(\mbb{R}^n)}\leq 3c_0\sigma(1+\norm{M}).
\end{equation}

Now Theorem \ref{thm:reg-nl-apndx} (e) applies to $w^\delta$ so that we get a constant $C_{\norm{M}}>1$ for which $w^\delta\in C^{2,\alpha}(\ov{B}_{C_{\norm{M}}^{-1}\sqrt{n}}(y_0))$ and 
\begin{equation*}
\norm{w^\delta}_{C^{2,\alpha}(\ov{B}_{C_{\norm{M}}^{-1}\sqrt{n}}(y_0))}^*\leq C_{\norm{M}}\left(\norm{w^\delta}_{L^\infty(B_{\sqrt{n}}(y_0))}+1\right)\leq\tilde{C}_{\norm{M}}\delta^{-1},
\end{equation*}
where $\norm{\cdot}_{C^{2,\alpha}(E)}^*$ is the adimensional $C^{2,\alpha}$ norm on $E$. Since $y_0\in\mbb{R}^n$ was an arbitrary point and $B_{\sqrt{n}}(y_0)$ contains a periodic cube of $w^\delta$, we obtain the estimate \eqref{eq:key-nl-d-C2a-dwd}. 
\end{proof}

Our next step is to find a uniform bound of the oscillation of $w^\delta$ for $\delta\in(0,1)$.

\begin{lem}\label{lem:key-nl-osc} Let $M\in\mc{S}^n$, $x\in\ov{\Omega}$ and $w^\delta$ be the unique solution to \eqref{eq:key-nl-d}. Then 
\begin{equation}\label{eq:key-nl-osc}
\sup_{0<\delta<1}\osc_{\mbb{R}^n}w^\delta\leq C(1+\norm{M}).
\end{equation}
Moreover, there holds 
\begin{equation}\label{eq:key-nl-d-C2a-twd}
\sup_{0<\delta<1}\norm{\tilde{w}^\delta}_{C^{2,\alpha}(\mbb{R}^n)}\leq C_{\norm{M}},
\end{equation}
where $\tilde{w}^\delta:=w^\delta-w^\delta(0)$ in $\mbb{R}^n$.
\end{lem}

\begin{proof} The proof follows the line of the proof of Lemma \ref{lem:key-l-osc}.
\end{proof}

It is noteworthy to observe that the derivatives of $w^\delta$ are bounded independent of $\delta\in(0,1)$. To be specific, since $Dw^\delta=D\tilde{w}^\delta$ and $D^2w^\delta=D^2\tilde{w}^\delta$, we obtain from \eqref{eq:key-nl-d-C2a-twd} that
\begin{equation}\label{eq:key-nl-d-D-wd}
\sup_{0<\delta<1}\left(\norm{Dw^\delta}_{L^\infty(\mbb{R}^n)}+\norm{D^2w^\delta}_{L^\infty(\mbb{R}^n)}+[D^2w^\delta]_{C^\alpha(\mbb{R}^n)}\right)\leq C_{\norm{M}}.
\end{equation}

We are now in position to prove Lemma \ref{lem:key-nl}. 

\begin{proof}[Proof of Lemma \ref{lem:key-nl}] One may notice that the proof of Lemma \ref{lem:key-l} has nothing to do with the linear structure of \eqref{eq:key-l}. Indeed, \eqref{eq:key-nl-d-linf-dwd} and \eqref{eq:key-nl-d-C2a-twd} respectively correspond to \eqref{eq:key-l-d-linf-dwd} and \eqref{eq:key-l-d-C2a-twd}. Hence, by the compact embedding, we are able to extract a subsequence $\{\delta_kw^{\delta_k},\tilde{w}^{\delta_k}\}_{k=1}^\infty$ from $\{\delta w^\delta,\tilde{w}^\delta\}_{0<\delta<1}$ such that
\begin{equation}\label{eq:key-nl-conv}
\norm{\delta_kw^{\delta_k}-\gamma}_{L^\infty(\mbb{R}^n)}+\norm{\tilde{w}^{\delta_k}-w}_{C^2(\mbb{R}^n)}\lra 0\quad\text{as}\quad k\lra\infty,
\end{equation}
for some $\gamma\in\mbb{R}$ and $w\in C^{2,\alpha}(\mbb{R}^n)$. In addition, we have that $
|\gamma|\leq\sigma(1+\norm{M})$ and $\norm{w}_{C^{2,\alpha}(\mbb{R}^n)}\leq C_{\norm{M}}$. The rest of the proof is exactly the same with that of Lemma \ref{lem:key-l} and hence is omitted.\end{proof} 

\begin{defn}\label{defn:wd/w/ovF-nl}
Let $(M,x)\in\mc{S}^n\times\ov{\Omega}$. 
\begin{enumerate}[(i)]
\item For each $\delta\in(0,1)$, we denote $w^\delta(\cdot;M,x)$ by the unique bounded 1-periodic solution of \eqref{eq:key-nl-d} and $\tilde{w}^\delta(\cdot;M,x)=w^\delta(\cdot;M,x)-w^\delta(0;M,x)$ in $\mbb{R}^n$. By the uniqueness of the solution, we can understand $w^\delta(y;\cdot,\cdot)$ as the mapping $(M,x)\mapsto w^\delta(y;M,x)$ defined on $\mc{S}^n\times\ov{\Omega}$ for each $y\in\mbb{R}^n$. 

\item In a similar way, we write $\ov{F}(M,x)$ by the unique number $\gamma$ of \eqref{eq:key-nl} and $w(\cdot;M,x)$ by the bounded 1-periodic solution of \eqref{eq:key-nl} which is normalized by $w(0;M,x)=0$. Again the uniqueness allows us to understand $\ov{F}$ [resp., $w(y;\cdot,\cdot)$ for each $y\in\mbb{R}^n$] as the mapping $(M,x)\mapsto\ov{F}(M,x)$ [resp., $w(y;M,x)$] defined on $\mc{S}^n\times\ov{\Omega}$.
\end{enumerate}
\end{defn}

Note that \eqref{eq:key-nl} now reads
\begin{equation}\label{eq:key-nl-re}
\begin{cases}
F(D_y^2w+M,x,y)=\ov{F}(M,x)&\text{in }\mbb{R}^n,\\
\text{$w$ is 1-periodic.}
\end{cases}
\end{equation}

The next lemma states that $\delta w^\delta$ and $\tilde{w}^\delta$ are locally Lipschitz continuous in $(M,x)$. One may also find a proof for \eqref{eq:lip-nl-dwd} in \cite{AB} and \cite{E2} regarding a more general situation. The proof for \eqref{eq:lip-nl-twd} can also be found in \cite{M} with a different argument.

\begin{lem}\label{lem:prop:nl-lip} For any $L>0$ and $(M,x),(M',x')\in\ov{B}_L\times\ov{\Omega}$, we have
\begin{align}
\label{eq:lip-nl-dwd}
&\norm{\delta w^\delta(\cdot;M',x')-\delta w^\delta(\cdot;M,x)}_{L^\infty(\R^n)}\leq C_L(\norm{M'-M}+|x'-x|),\end{align}
and 
\begin{align}
\label{eq:lip-nl-twd}
&\norm{\tilde{w}^\delta(\cdot;M',x')-\tilde{w}^\delta(\cdot;M,x)}_{L^\infty(\R^n)}\leq C_L(\norm{M'-M}+|x'-x|).
\end{align}
\end{lem}

\begin{proof} For brevity, let us denote by $v_1^\delta$ [resp., $v_2^\delta$] the function $w^\delta(\cdot;M',x')$ [resp., $w^\delta(\cdot;M,x)$]. Also by $\tilde{v}_1^\delta$ [resp., $\tilde{v}_2^\delta$] we denote $\tilde{w}^\delta(\cdot;M',x')$ [resp., $\tilde{w}^\delta(\cdot;M,x)$]. 

We prove \eqref{eq:lip-nl-dwd} first. By the Lipschitz continuity of $F$, we get
\begin{equation*}
F(D_y^2v_2^\delta+M,x,y)\geq\delta v_2^\delta-\sigma(1+L)(\norm{M'-M}+|x'-x|)
\end{equation*}
which implies that $v_2^\delta-\delta^{-1}\sigma(1+L)(\norm{M'-M}+|x'-x|)$ is a subsolution of \eqref{eq:key-nl-d}. By the comparison principle (Theorem \ref{thm:exist-nl-apndx}), we arrive at
\begin{equation*}
\delta v_2^\delta-\delta v_1^\delta\leq\sigma(1+L)(\norm{M'-M}+|x'-x|)\quad\text{in }\mbb{R}^n.
\end{equation*}
By a similar argument, we obtain \eqref{eq:lip-nl-dwd} with $C_L\geq\sigma(1+L)$.

Now we move on to the proof of \eqref{eq:lip-nl-twd}. The main idea is to use the linearisation of $F$. Define $a_{ij}^\delta=\int_0^1F_{p_{ij}}(N_t^\delta,x_t,\cdot)dt$ and $b_k^\delta=\int_0^1F_{x_k}(N_t^\delta,x_t,\cdot)dt$ where $N_t^\delta:=t\{D^2v_1^\delta+M'\}+(1-t)\{D^2v_2^\delta+M\}$ and $x_t:=tx+(1-t)x'$. It is immediate from the structure conditions (F1)-(F3) that $a_{ij}^\delta$ and $b_k^\delta$ ($i,j,k=1,\ldots,n$) are 1-periodic and uniformly bounded in $\R^n$ by the Lipschitz constant of $F$. Furthermore, $(a_{ij}^\delta)$ is uniformly elliptic with the same ellipticity constants $\lambda$ and $\Lambda$ of $F$.

Now define $v^\delta:=v_1^\delta-v_2^\delta$ and $\tilde{v}^\delta:=\tilde{v}_1^\delta-\tilde{v}_2^\delta$. Then $v^\delta,\tilde{v}^\delta\in C^{2,\alpha}(\R^n)$ solve
\begin{equation}\label{eq:key-nl-d-lin}
a_{ij}^\delta D_{ij}w+a_{ij}^\delta(M_{ij}'-M_{ij})+b_k^\delta(x_k'-x_k)=\delta v^\delta\quad\text{in }\R^n.
\end{equation}
As this equation belongs to the same class of \eqref{eq:key-l-d}, we arrive the conclusion by the same argument used in Lemma \ref{lem:key-l-osc}. We left the details to the readers.
\end{proof}

\begin{lem}\label{lem:prop:nl-uconv} The convergence in \eqref{eq:key-nl-conv} is uniform in $(M,x)\in\ov{B}_L\times\ov{\Omega}$ for each $L>0$; i.e.,
\begin{equation}\label{eq:prop:nl-uconv}
\begin{split}
&\lim_{\delta\ra 0}\sup_{(M,x)\in\ov{B}_L\times\ov{\Omega}}\norm{\delta w^\delta(\cdot;M,x)-\ov{F}(M,x)}_{L^\infty(\mbb{R}^n)}=0,\\
\end{split}
\end{equation}
and
\begin{equation}\label{eq:prop:nl-uconv'}
\begin{split}
&\lim_{\delta\ra 0}\sup_{(M,x)\in\ov{B}_L\times\ov{\Omega}}\norm{\tilde{w}^\delta(\cdot;M,x)-w(\cdot;M,x)}_{C^2(\mbb{R}^n)}=0.
\end{split}
\end{equation}
\end{lem}

\begin{proof} Fix $L>0$. Put $C_L=\sup\{C_{\norm{M}}:M\in\ov{B}_L\}$ and then take $\tilde{C}_L=\max\{\sigma(1+L),C_L\}$. Then it follows from \eqref{eq:key-nl-d-linf-dwd} and \eqref{eq:key-nl-d-C2a-twd} that for any $\delta\in(0,1)$,
\begin{equation}\label{eq:key-nl-est-dwd/twd}
\sup_{(M,x)\in\ov{B}_L\times\ov{\Omega}}\left\{\norm{\delta w^\delta(\cdot;M,x)}_{L^\infty(\mbb{R}^n)},\norm{\tilde{w}^\delta(\cdot;M,x)}_{C^{2,\alpha}(\mbb{R}^n)}\right\}\leq\tilde{C}_L.
\end{equation}
The above uniform estimates allow us to extract a subsequence $\{\delta_k w^{\delta_k}\}_{k=1}^\infty$ [resp. $\{\tilde{w}^{\delta_k}\}_{k=1}^\infty$] from $\{\delta w^\delta\}_{0<\delta<1}$ [resp. $\{\tilde{w}^\delta\}_{0<\delta<1}$] such that \eqref{eq:key-nl-conv} holds regardless of a particular choice of $(M,x)\in\ov{B}_L\times\ov{\Omega}$. The rest of the proof is the same with that in Lemma \ref{lem:key-nl}.
\end{proof}

It is an immediate consequence of Lemma \ref{lem:prop:nl-lip} and \ref{lem:prop:nl-uconv} that the effective operator $\ov{F}$ and the corresponding corrector $w(y;\cdot,\cdot)$ are locally Lipschitz continuous (uniform in $y$). Due to its particular role in the rest of this paper, we present the statement without proof.

\begin{lem}\label{lem:prop:nl-lip-h} $\ov{F}$ and $w(y;\cdot,\cdot)$ are Lipschitz continuous locally in $\mc{S}^n$ and globally in $\ov{\Omega}$. Moreover, the Lipschitz continuity of the latter is uniform in $y\in\R^n$.
\end{lem}

There are additional properties of $\ov{F}$. A more general proof is contained in \cite{E2}. Here we make a slight adjustment of the proof according to our situation; the main difference is that we have $C^{2,\alpha}$-corrector, which makes the proof simpler.

\begin{lem}\label{lem:prop:nl} 
\begin{enumerate}[(i)]
\item $\ov{F}$ is uniformly elliptic with the same constants $\lambda$ and $\Lambda$ of $F$.
\item $\ov{F}$ is concave on $\mc{S}^n$.
\end{enumerate}
\end{lem}

\begin{proof} The proof for the assertion (i) is similar to that of the assertion (ii) of Lemma \ref{lem:eff op-l}, so is omitted. 

Now we establish the proof of (ii). Let $M,N\in\mc{S}^n$ and $x\in\ov{\Omega}$ be given. For simplicity let us write $w^M$ by the solutions of \eqref{eq:key-nl-re} with respect to $M$. 

Suppose toward a contradiction that there is some $t\in(0,1)$ and $M,N\in\mc{S}^n$ such that
\begin{equation*}
\ov{F}(tM+(1-t)N,x)<t\ov{F}(M,x)+(1-t)\ov{F}(N,x).
\end{equation*}
Put $X:=tM+(1-t)N\in\mc{S}^n$. Adding a constant to $w^X$ if necessary, we may assume that $w^X<tw^M+(1-t)w^N$ in $\mbb{R}^n$. Then we obtain from the concavity of $F$ that
\begin{equation*}
\ov{F}(X,x)<t\ov{F}(M,x)+(1-t)\ov{F}(N,x)\leq F(X+D_y^2(tw^M+(1-t)w^N),x,y)
\end{equation*}
in $\mbb{R}^n$. However, since $\ov{F}(X+D_y^2w^X,x,y)=\ov{F}(X,x)$ in $\R^n$, the comparison principle implies that $w^X\geq tw^M+(1-t)w^N$ in $\mbb{R}^n$, which is a contradiction.
\end{proof}

As we mentioned in the beginning of this section, we determine the equation which $u$ solves in the viscosity sense.

\begin{lem}\label{lem:main-nl-h} Assume that $F\in C(\mc{S}^n\times\ov{\Omega}\times\mbb{R}^n)$ satisfy the hypotheses (F1)-(F4). Then the function $u$ from Lemma \ref{lem:limit-nl} solves 
\begin{equation}\tag{$\ov{F}$}\label{eq:nl-h}
\begin{cases}
\ov{F}(D^2u,x)=0&\text{in }\Omega,\\
u=g&\text{on }\p\Omega.
\end{cases}
\end{equation}
Moreover, $u$ is unique and belongs to the class of $C^{2,\alpha}(\ov{\Omega})$.
\end{lem}

\begin{proof} The proof of that $u$ is a viscosity solution of \eqref{eq:nl-h} is similar to that of Lemma \ref{lem:main-l-h}. Instead of using strong maximum principle, one may take advantage of Theorem \ref{thm:exist-nl-apndx} (a). The details are left to the readers.

As long as we know that $u$ solves \eqref{eq:nl-h}, the fact that $u\in C^{2,\alpha}(\Omega)$ follows readily from Theorem \ref{thm:reg-nl-apndx} (e). The proof is similar to that in Lemma \ref{lem:key-nl-d}, so the details are omitted; instead of taking advantage of (F1)-(F4), we use Lemma \ref{lem:prop:nl} (i)-(iii). We make a remark here that the exponent $\alpha$ is the same with which we chose in Lemma \ref{lem:key-nl-d} because the ellipticity constants of $\ov{F}$ coincide with those of $F$ (Lemma \ref{lem:prop:nl} (i)).
\end{proof}

\subsection{Regularity of the effective operator and the corrector}\label{sec:reg-nl}

In the previous subsection, we observed that the Lipschitz regularity of $F$, in particular in the $(M,x)$-variable, yields the Lipschitz regularity of $\ov{F}$ and $w(y;\cdot,\cdot)$, where the regularity for the latter is uniform in $y\in\R^n$. Then, it is natural to ask whether higher regularity of $F$ in $(M,x)$-variable gives higher regularity for $\ov{F}$ and $w(y;\cdot,\cdot)$, and we prove in this subsection that the answer is affirmative. Specifically, we observe that they have the same regularity as $F$ does. This regularity result plays the key role in the rest of this paper, especially in seeking higher order interior correctors. To be precise, we observe the following. 

\begin{prop}\label{prop:reg-nl} $\ov{F}$ and $w(y;\cdot,\cdot)$ are $C^{m,1}$ locally in $\mc{S}^n$ and globally in $\ov{\Omega}$ and for any $L>0$, 
\begin{equation}\label{eq:reg-nl}
\norm{\ov{F}}_{C^{m,1}(\ov{B}_L\times\ov{\Omega})}+\norm{w(y,\cdot,\cdot)}_{C^{m,1}(\ov{B}_L\times\ov{\Omega})}\leq C_{L,m}
\end{equation}
Moreover, for any $(M',x'),(M,x)\in\ov{B}_L\times\ov{\Omega}$ there holds
\begin{equation}\label{eq:reg-nl-C2a}
\begin{split}
\sum_{0\leq i+j\leq m-1}&\norm{D_p^iD_x^jw(\cdot;M',x')-D_p^iD_x^jw(\cdot;M,x)}_{C^{2,\alpha}(\R^n)}\\
&\quad\quad\quad\quad\leq C_{L,m}(\norm{M'-M}+|x'-x|).
\end{split}
\end{equation}
\end{prop}

\begin{rem} Note that the estimate \eqref{eq:reg-nl-C2a} implies that $D_y^iw(y;\cdot,\cdot)\in C^{m-1,1}(\ov{B}_L\times\ov{\Omega})$ for $i=1,2$. This will turn out as the coupling effect as we mentioned in Sect. \ref{sec:intro}.
\end{rem}

Before we begin the proof, let us illustrate the heuristics of our argument. In the first place, we only assume that $F$ satisfies the structure condition (F3) with $m=1$, which means that $F$ is $C^{1,1}$ locally in $\mc{S}^n$ and globally in $\ov{\Omega}\times\R^n$, and arrive at the conclusion that $\ov{F}$ and $w(y;\cdot,\cdot)$ are also $C^{1,1}$ locally in $\mc{S}^n$ and globally in $\ov{\Omega}$. We also observe that the equation, which involves the partial derivatives of $\ov{F}$ and $w(y;\cdot,\cdot)$ in $M$ and $x$-variable, satisfies the same structure conditions of $F$. This implies that under our original assumption (F3) we are able to iterate the argument to get $C^{m,1}$ regularity of $\ov{F}$ and $w(y;\cdot,\cdot)$ which is local in $\mc{S}^n$ and global in $\ov{\Omega}$.  

As the first step, we prove that if $F\in C^{1,1}$, then the $L^\infty$-norm in \eqref{eq:lip-nl-dwd} and \eqref{eq:lip-nl-twd} can be improved by $C^{2,\alpha}$-norm. 

\begin{lem}\label{lem:prop:nl-lip-C2a} For each $L>0$ and $(M,x),(M',x')\in\ov{B}_L\times\ov{\Omega}$, there hold for all $\delta\in(0,1)$,
\begin{equation}\label{eq:lip-nl-C2a-dwd}
\norm{\delta w^\delta(\cdot;M',x')-\delta w^\delta(\cdot;M,x)}_{C^{2,\alpha}(\mbb{R}^n)}\leq C_L(\norm{M'-M}+|x'-x|)
\end{equation}
and 
\begin{equation}\label{eq:lip-nl-C2a-twd}
\norm{\tilde{w}^\delta(\cdot;M',x')-\tilde{w}^\delta(\cdot;M,x)}_{C^{2,\alpha}(\mbb{R}^n)}\leq C_L(\norm{M'-M}+|x'-x|).
\end{equation}
\end{lem}

\begin{proof} The main idea has been already introduced in the proof of Lemma \ref{lem:prop:nl-lip}. We only need to obtain a uniform $C^{0,\alpha}(\R^n)$-estimate on the linearized coefficients $a_{ij}^\delta$ and $b_k^\delta$; recall all the notations used in Lemma \ref{lem:prop:nl-lip}. Here we only present the proof for $a_{ij}^\delta$, since that of $b_k^\delta$ follows the same argument.

By the estimate \eqref{eq:key-nl-d-D-wd}, we have that for any $t\in[0,1]$, $\norm{N_t^\delta}_{L^\infty(\R^n)}\leq C_L+L$. Hence, we deduce from the condition (F3) that $
\norm{a_{ij}^\delta}_{L^\infty(\R^n)}\leq\sigma(C_L+L+1)$. Again by \eqref{eq:key-nl-d-D-wd}, for any $y_1,y_2\in Q_1$, $\norm{N_t^\delta(y_1)-N_t^\delta(y_2)}\leq C_L|y_1-y_2|^\alpha$. Thus, the periodicity of $a_{ij}^\delta$ yields that $[a_{ij}^\delta]_{C^{0,\alpha}(\R^n)}\leq\tilde{C}_L$, where $\tilde{C}_L=\sigma(C_L+L+1)(C_L+1)$. Summing up we get that $\norm{a_{ij}^\delta}_{C^{0,\alpha}(\R^n)}\leq 2\tilde{C}_L$.

It is also easy to see that $\norm{\delta v^\delta}_{C^{0,\alpha}(\R^n)}\leq 6c_0\sigma(1+L)$. Therefore, we may apply the interior Schauder estimate to \eqref{eq:key-nl-d-lin} in a ball $B_{\sqrt{n}}$ containing a periodic cube to get the conclusion, as in Lemma \ref{lem:key-l-d} and \ref{lem:key-l-osc}.
\end{proof}

As a corollary, we obtain the same Lipschitz continuity of $w(y;\cdot,\cdot)$ in $(M,x)$-variable which is uniform in the $C^{2,\alpha}(\R^n)$-norm.

\begin{lem}\label{lem:prop:nl-lip-C2a-h} For each $L>0$ and $(M,x),(M',x')\in\ov{B}_L\times\ov{\Omega}$, there holds
\begin{equation}\label{eq:lip-nl-C2a-w}
\norm{w(\cdot;M',x')-w(\cdot;M,x)}_{C^{2,\alpha}(\mbb{R}^n)}\leq C_L(\norm{M'-M}+|x'-x|).
\end{equation}
\end{lem}

\begin{proof} Apply the uniform convergence (Lemma \ref{lem:prop:nl-uconv}) to get 
\begin{equation*}
\norm{w(\cdot;M',x')-w(\cdot;M,x)}_{C^2(\mbb{R}^n)}\leq C_L(\norm{M'-M}+|x'-x|).
\end{equation*}
Then use the uniform boundedness of $C^{2,\alpha}(\R^n)$-norm of $w(\cdot;M',x')-w(\cdot;M,x)$ (Lemma \ref{lem:key-nl}) and the compactness embedding to improve this inequality to $C^{2,\alpha}(\R^n)$-norm. 
\end{proof}

In the subsequent two lemmas, we show that $\ov{F}$ and $w(y;\cdot,\cdot)$ are differentiable and further that the partial derivatives are locally Lipschitz continuous on $\mc{S}^n\times\ov{\Omega}$. The former is done by linearizing the equation \eqref{eq:key-nl-re}. In order to get the latter, however, we need to begin our argument from the linearized equation \eqref{eq:key-nl-d-lin}. 

\begin{lem}\label{lem:pd-nl-e} There exist $\ov{F}_{p_{kl}},\ov{F}_{x_k}$, $D_{p_{kl}}w(y;\cdot,\cdot)$ and $D_{x_k}w(y;\cdot,\cdot)$ for each $y\in\R^n$ on $\mc{S}^n\times\ov{\Omega}$. In addition, there hold for any $L>0$ and $(M,x)\in\ov{B}_L\times\ov{\Omega}$,
\begin{equation}\label{eq:pd-nl-bd}
|\ov{F}_{p_{kl}}(M,x)|,|\ov{F}_{x_k}(M,x)|,\norm{D_{p_{kl}}w(\cdot;M,x)}_{C^{2,\alpha}(\R^n)},\norm{D_{x_k}w(\cdot;M,x)}_{C^{2,\alpha}(\R^n)}\leq C_L.
\end{equation}
\end{lem}

\begin{proof} Here we only provide the proof for the $M$-partial derivatives of $\ov{F}$ and $w(y;\cdot,\cdot)$. The argument for the $x$-partial derivatives is similar so we omit it to avoid the redundancy.

Pick any $L>0$ and $(M,x)\in\ov{B}_L\times\ov{\Omega}$. By $v_h$ we denote $h^{-1}[w(\cdot;M+hE^{kl},x)-w(\cdot;M,x)]$. As we linearize the equation \eqref{eq:key-nl-re} with $M+hE^{kl}$ and $M$, and divide the both sides by $h$, we observe that $v_h$ satisfies
\begin{equation}\label{eq:lin-nl-h}
a_{ij,h}D_{ij}v_h+a_{kl,h}=\gamma_h
\end{equation}
where $a_{ij,h}:=\int_0^1F_{p_{ij}}(N_{t,h},x,\cdot)dt$, $\gamma_h:=h^{-1}[\ov{F}(M+hE^{kl},x)-\ov{F}(M,x)]$ and $N_{t,h}:=tD_y^2w(\cdot;M+hE^{kl},x)+(1-t)D_y^2w(\cdot;M,x)+M+thE^{kl}$.

By following the argument in the proof of Lemma \ref{lem:prop:nl-lip-C2a}, we observe that for any $h$ with $|h|$ small, $a_{ij,h}$ is also uniformly elliptic with the ellipticity constants $\lambda$ and $\Lambda$, and belongs to $C^{0,\alpha}(\R^n)$ with $\norm{a_{ij,h}}_{C^{0,\alpha}(\R^n)}\leq c_L$. Also we know from Lemma \ref{lem:prop:nl-lip-h} that $|\gamma_h|\leq\tilde{c}_L$.

Therefore, the linearized equation \eqref{eq:lin-nl-h} belongs to the same class of \eqref{eq:key-l-d}. Even though the coefficients of \eqref{eq:lin-nl-h} vary with respect to the parameter $h$, the proof of Lemma \ref{lem:key-l} is still applicable because we have a uniform convergence of $a_{ij,h}$ as $h\ra 0$; indeed, Lemma \ref{lem:prop:nl-lip-C2a-h} implies that $a_{ij,h}\ra a_{ij}:=F_{p_{ij}}(D_y^2w(\cdot;M,x)+M,x,\cdot)$ uniformly in $\R^n$ as $h\ra 0$. Consequently, there exist a unique constant $\gamma$ and a bounded 1-periodic function $v\in C^{2,\alpha}(\mbb{R}^n)$ such that
\begin{equation*}
|\gamma_h-\gamma|+\norm{v_h-v}_{C^2(\mbb{R}^n)}\lra 0
\end{equation*}
as $h\ra 0$ and that $v$ satisfies
\begin{equation}\label{eq:lin-nl}
a_{ij}D_{ij}v+a_{kl}=\gamma\quad\text{in }\mbb{R}^n.
\end{equation}
By the convergence above, $\gamma=\ov{F}_{p_{kl}}(M,x)$ and $v=D_{p_{kl}}w(\cdot;M,x)$. One should notice that we do not force $v(0)$ to be 0 here; otherwise, we could not say that $v=D_{p_{kl}}w(\cdot;M,x)$. The uniform estimate \eqref{eq:pd-nl-bd} now follows from Lemma \ref{lem:prop:nl-lip-h} and \ref{lem:prop:nl-lip-C2a-h}.
\end{proof}

\begin{lem}\label{lem:pd-nl-lip} $\ov{F}_{p_{kl}}$, $\ov{F}_{x_k}$, $D_{p_{kl}}w(y;\cdot,\cdot)$ and $D_{x_k}w(y;\cdot,\cdot)$ are Lipschitz continuous locally in $\mc{S}^n$ and globally in $\ov{\Omega}$. Moreover, the Lipschitz continuity of the latter two is uniform $y\in\R^n$.  
\end{lem}

\begin{proof} Here we only present the proof for the $M$-partial derivatives. The proof for the $x$-partial derivatives is the same, and we leave it to the readers. 

Substituting $M'$ [resp., $x'$] with $M+hE^{kl}$ [resp., $x$] in the equation \eqref{eq:key-nl-d-lin} and dividing by $h$ the both sides, one obtains
\begin{equation}\label{eq:lin-nl-d-h}
a_{ij,h}^\delta D_{ij}v_h^\delta+a_{kl,h}^\delta-\delta v_h^\delta=0\quad\text{in }\R^n,
\end{equation}
where $a_{ij,h}^\delta:=\int_0^1F_{p_{ij}}(N_{t,h}^\delta,x,\cdot)dt$, $v_h^\delta:=h^{-1}[w^\delta(\cdot;M+hE^{kl},x)-w^\delta(\cdot;M,x)]$ and $N_{t,h}^\delta:=tD_y^2w^\delta(\cdot;M+hE^{kl},x)+(1-t)D_y^2w^\delta(\cdot;M,x)+M+thE^{kl}$. 

By Lemma \ref{lem:prop:nl-lip-C2a}, we have $\norm{v_h^\delta}_{C^{2,\alpha}(\R^n)}\leq C_L$ for any $0<|h|<1$ and $\delta>0$. Then the Arzela-Ascoli theorem yields that for each $\delta>0$, there is a bounded 1-periodic $v^\delta\in C^{2,\alpha}(\R^n)$ such that $v_h^\delta\ra v^\delta$ in $C^2(\R^n)$ along a subsequence of $h$. Moreover, this lemma implies that $a_{ij,h}^\delta\ra a_{ij}^\delta:=F_{p_{ij}}(D_y^2w^\delta(\cdot;M,x)+M,x,\cdot)$ uniformly in $\R^n$ as $h\ra 0$. Since $a_{ij}^\delta$ is also uniformly elliptic with the same ellipticity constants $\lambda$ and $\Lambda$, the stability of the viscosity solutions (c.f. the proof of Lemma \ref{lem:key-l}) then ensures that the limit function $v^\delta$ solves
\begin{equation}\label{eq:lin-nl-d}
a_{ij}^\delta D_{ij}v^\delta+a_{kl}^\delta-\delta v^\delta=0\quad\text{in }\R^n.
\end{equation}
Due to the uniqueness of the solution of \eqref{eq:lin-nl-d} (c.f. Lemma \ref{lem:key-l-d}), we now know that $v_h^\delta\ra v^\delta$ in $C^2(\R^n)$ as $h\ra 0$; i.e., the convergence is valid for the full sequence of $h$.
  
From now on we write $a_{ij}^\delta=a_{ij}^\delta(\cdot;M,x)$ [resp., $v^\delta=v^\delta(\cdot;M,x)$] to specify the dependency on $(M,x)$. We claim that the equation \eqref{eq:lin-nl-d} is a $\delta$-penalization of the equation \eqref{eq:lin-nl}; i.e., the limit of the normalized function $\tilde{v}^\delta(\cdot;M,x):=v^\delta(\cdot;M,x)-v^\delta(0;M,x)$ solves the equation \eqref{eq:lin-nl}. It is enough to prove that $a_{ij}^\delta(\cdot;M,x)\ra a_{ij}(\cdot;M,x)=F_{p_{ij}}(D_y^2w(\cdot;M,x)+M,x,\cdot)$ uniformly in $\R^n$ as $\delta\ra 0$, since then the rest of the proof follows the lines of Lemma \ref{lem:key-nl}. However, by Lemma \ref{lem:prop:nl-uconv} and \ref{lem:prop:nl-lip-C2a}, we have 
\begin{equation*}
\lim_{(\delta,h)\ra(0+,0)}\sup_{(M,x)\in\ov{B}_L\times\ov{\Omega}}\norm{a_{ij,h}^\delta(\cdot;M,x)-a_{ij}(\cdot;M,x)}_{L^\infty(\R^n)}=0,
\end{equation*}
which gives the desired convergence. 

Next we claim that for each $L>0$, $a_{ij}^\delta(y;\cdot,\cdot)$ is Lipschitz continuous in $\ov{B}_L\times\ov{\Omega}$ uniformly for $y\in\R^n$ and $\delta\in(0,1)$. If so, then we arrive at our conclusion by applying Lemma \ref{lem:prop:nl-lip}, since the equations \eqref{eq:lin-nl-d} and \eqref{eq:key-nl-d} are in the same class. 

To see this, choose any $L>0$ and $(N,z),(N',z')\in\ov{B}_L\times\ov{\Omega}$. According to \eqref{eq:key-nl-d-D-wd}, the $C^{2,\alpha}(\R^n)$-norm of both $w^\delta(\cdot;N,z)$ and $w^\delta(\cdot;N',z')$ is uniformly bounded by $C_L$. Thus, the structure condition (F3) together with \eqref{eq:lip-nl-C2a-dwd} yields that
\begin{equation*}
\begin{split}
\norm{a_{ij}^\delta(\cdot;N,z)-a_{ij}^\delta(\cdot;N',z')}_{L^\infty(\R^n)}&\leq\tilde{C}_L(\norm{N-N'}+|z-z'|),
\end{split}
\end{equation*}
where $\tilde{C}_L=C_L\sigma(1+C_L)$, which proves the claim.
\end{proof}

\begin{rem} Note that the limit of the normalized function $\tilde{v}^\delta(\cdot;M,x)$ may not be equal to $D_{p_{kl}}w(\cdot;M,x)$, since we cannot assure that $D_{p_{kl}}w(0;M,x)=0$. In fact, those two functions differ by an additive constant. It is the main reason why we do not use the $\delta$-penalization argument to derive Lemma \ref{lem:pd-nl-e}, although the proofs are essentially the same. 
\end{rem}

We are now in position to present the proof of our main proposition of this subsection.

\begin{proof}[Proof of Proposition \ref{prop:reg-nl}] Observe from Lemma \ref{lem:pd-nl-lip} the first order partial derivatives of $\ov{F}$ and $w(y;\cdot,\cdot)$ satisfies the equations (e.g., \eqref{eq:lin-nl}) which belong to the same class of \eqref{eq:key-nl}, and admit the $\delta$-approximating problems (e.g., \eqref{eq:lin-nl-d}) which correspond to \eqref{eq:key-nl-d}. Thus, we can repeat the argument used through Lemma \ref{lem:prop:nl-lip-C2a}-\ref{lem:pd-nl-lip} again to get the Lipschitz continuity of the second order partial derivatives of $\ov{F}$ and $w(y;\cdot,\cdot)$. We iterate this process by $m$-times to reach the conclusion. We leave the details to the readers.
\end{proof}

\subsection{Interior and boundary layer correctors}\label{sec:int/bdry cor&rate-nl}

Now we are in position to construct higher order correctors which correct the error occurring in the interior and on the boundary layer of our physical domain $\Omega$. This subsection involves many iterative arguments, so before we make our argument rigorous, we would like to provide the key idea.

First and foremost, we emphasize that the asymptotic expansion of $u^\e$ occurs inside of the operator $F$, which differs from the linear case. That is, if $\eta_r^\e:=u+\sum_{k=1}^r\e^kw_k(\e^{-1}x,x)$ is our expansion, then after a computation we get
\begin{equation*}
\begin{split}
F\left(D^2\eta_r^\e,x,\frac{x}{\e}\right)=F\left(X^0+\e Y^r,\frac{x}{\e},x\right)
\end{split}
\end{equation*}
where
\begin{equation}\label{eq:Xk-nl}
X^k=\begin{cases}
D_x^2u(\cdot)+D_y^2w_2(\cdot/\e,\cdot)&\text{if }k=0\\
D_x^2w_k(\cdot/\e,\cdot)+D_{x,y}w_{k+1}(\cdot/\e,\cdot)+D_y^2w_{k+2}(\cdot/\e,\cdot)&\text{if }1\leq k\leq r-2,\\
D_x^2w_{r-1}(\cdot/\e,\cdot)+D_{x,y}w_r(\cdot/\e,\cdot)&\text{if }k=r-1,\\
D_x^2w_r(\cdot/\e,\cdot)&\text{if }k=r,
\end{cases}
\end{equation}
and $Y^r$ defined by
\begin{equation}\label{eq:Yr-nl}
Y^r=X^1+\e X^2+\cdots+\e^{r-1}X^r. 
\end{equation}
Here we have denoted $D_xD_y+D_yD_x$ by $D_{x,y}$. To further simplify our notation, let us drop the dependency of $(\e^{-1}x,x)$. Then a Taylor expansion of $F$ with respect to the Hessian gives, 
\begin{equation*}
\begin{split}
F(X^0+\e Y^r)&=F(X^0)+\e F_{p_{ij}}(X^0)Y_{ij}^r+\cdots+\frac{\e^r}{r!}F_{p_{i_1j_1}\ldots p_{i_rj_r}}(X^0)Y_{i_1j_1}^r\cdots Y_{i_rj_r}^r\\
&\quad+O(\e^{r+1}),
\end{split}
\end{equation*}
which would be valid provided that $\norm{Y^r}_{L^\infty(\Omega)}\leq C$ with a positive constant independent of $\e$. This in turn requires us to have a uniform control (i.e., independent of $\e$) on the supremum norm of second order derivatives of $w_k$ in both $x$ and $y$-variables.

Moreover, one should note that $Y^r=\sum_{k=1}^r\e^{k-1}X^k$ is a summation of the terms of different $\e$-order. For this reason we rearrange the terms in the Taylor expansion according to the $\e$-power as below. 
\begin{equation}\label{eq:te-nl}
\begin{split}
F(X^0+\e Y^r)&=F(X^0)+\e F_{p_{ij}}(X^0)X_{ij}^1+\cdots\\
&\quad+\e^r\sum_{l=1}^r\frac{1}{l!}\sum_{n_1+\cdots+n_l=r}F_{p_{i_1j_1}\ldots p_{i_lj_l}}(X^0)X_{i_1j_1}^{n_1}\cdots X_{i_lj_l}^{n_l}\\
&\quad+\sum_{l=1}^r\sum_{r+1\leq n_1+\cdots+n_l\leq rl}\frac{\e^{n_1+\cdots+n_l}}{l!}F_{p_{i_1j_1}\ldots p_{i_lj_l}}(X^0)X_{i_1j_1}^{n_1}\cdots X_{i_lj_l}^{n_l}\\
&\quad+O(\e^{r+1}).
\end{split}
\end{equation}
It suggests us to find $w_1,\ldots,w_r$ in such a way that $F(X^0)=0$, $F_{p_{ij}}(X^0)X_{ij}^1=0$, and so on. 

To satisfy $F(X^0)=0$, $w_2$ must be chosen such that $D_y^2w_2=D_y^2w(\cdot;D_x^2u,x)$. Then $F(X^0)=\ov{F}(D_x^2u)=0$ by Lemma \ref{lem:main-nl-h}. Furthermore, one should obtain, for $k=1,\ldots,r-2$, 
\begin{equation}\label{eq:int cor-nl-re}
\begin{split}
0&=\sum_{l=1}^k\frac{1}{l!}\sum_{n_1+\cdots+n_l=k}F_{p_{i_1j_1}\ldots p_{i_lj_l}}(X^0)X_{i_1j_1}^{n_1}\cdots X_{i_lj_l}^{n_l}\\
&=F_{p_{ij}}(X^0)X_{ij}^k+\sum_{l=2}^k\frac{1}{l!}\sum_{n_1+\cdots+n_l=k}F_{p_{i_1j_1}\ldots p_{i_lj_l}}(X^0)X_{i_1j_1}^{n_1}\cdots X_{i_lj_l}^{n_l}\\
&=F_{p_{ij}}(X^0)D_{y_iy_j}w_{k+2}+\Phi_{k+2},
\end{split}
\end{equation}
which yields the equation for $w_k$, where 
\begin{equation*}
\begin{split}
\Phi_{k+2}&=F_{p_{ij}}(X^0)D_{x_ix_j}w_k+2F_{p_{ij}}(X^0)D_{x_iy_j}w_{k+1}\\
&\quad+\sum_{l=2}^k\frac{1}{l!}\sum_{n_1+\cdots+n_l=k}F_{p_{i_1j_1}\ldots p_{i_lj_l}}(X^0)X_{i_1j_1}^{n_1}\cdots X_{i_lj_l}^{n_l}.
\end{split}
\end{equation*}
Notice that the summation on the right hand side involves $X^l$ for $l\leq k-1$ only; in other words, the term $\Phi_{k+2}$ has nothing to do with the functions $w_r$ with $r\geq k+2$. Thus, we are able to obtain $w_{k+2}$ by solving the equation \eqref{eq:int cor-nl-re} as long as $\Phi_{k+2}$ satisfies certain inductive hypotheses. On the other hand, since $w_{k+2}$ makes the $\e^k$-th order term in \eqref{eq:te-nl} to vanish, there is no opportunity to kill the $\e^{r-1}$ and $\e^r$-th order terms; recall that the same situation has happened in the linear setting. This in turn suggests that we can have at most 
\begin{equation*}
F(X^0+\e Y^r)=O(\e^{r-1}),
\end{equation*}
which would lead us to $O(\e^{r-1})$-rate of convergence (Theorem \ref{thm:main-nl}). Finally we make a remark that as in the linear case, we would come up with the compatibility condition of $w_{k+2}$, which determines uniquely $w_k$. Unlike the linear case (Lemma \ref{lem:chi-l}), however, this relationship is more hidden in the induction argument. We will discuss this issue in the proof in more detail.

Now we make our argument rigorous. Throughout this subsection we set $m\geq 2$. First we enhance the regularity of $u$, since now we have $\ov{F}\in C^{m,1}$. 

\begin{lem}\label{lem:u-nl-Cma} Assume that $F$ verifies the hypotheses (F1)-(F4). Then $u\in C^{m+2,\alpha}(\ov{\Omega})$ and 
\begin{equation*}
\norm{u}_{C^{m+2,\alpha}(\ov{\Omega})}\leq C_{m,g,\Omega}.
\end{equation*}
\end{lem}

\begin{proof} By Proposition \ref{prop:reg-nl} we know that $F$ is $C^{1,1}$ locally in $\mc{S}^n$ and globally in $\ov{\Omega}$. Since $u$ solves \eqref{eq:nl-h} where $g\in C^{m+2,1}(\ov{\Omega})$ and $\p\Omega\in C^{m+2,1}$, the regularity theory (Theorem \ref{thm:reg-nl-apndx} (f)) implies that $u\in C^{m+2,\alpha}(\ov{\Omega})$ and
\begin{equation*}
\norm{u}_{C^{m+2,\alpha}(\ov{\Omega})}\leq C_{\ov{F},\Omega}(\norm{u}_{L^\infty(\Omega)}+\norm{g}_{C^{m+2,1}(\ov{\Omega}}),
\end{equation*}
where $C_{\ov{F},\Omega}$ is a constant depending only on the derivatives of $\ov{F}$ up to $m$-th order, and on $\Omega$. By \eqref{eq:reg-nl}, $C_{\ov{F},\Omega}$ in turn depends only on the constants appearing in the structure conditions (F1)-(F4) and $m$. By an a priori estimate, on the other hand, we may bound the supremum norm of $u$ by a constant depending only on $\lambda,\Lambda,\Omega$ and $\norm{g}_{L^\infty(\Omega)}$. It completes the proof.
\end{proof}

Next we construct the interior higher order correctors. The regularity theory established in Subsection \ref{sec:reg-nl} now plays an essential role in proving the existence of the correctors and obtaining a uniform control on $L^\infty$-bound of their second order derivatives.

\begin{lem}\label{lem:int cor-nl}
Suppose $m\geq 2$. Then there exist a family of non-trivial 1-periodic functions $\{w_k:\mbb{R}^n\times\ov{\Omega}\ra\mbb{R}\}_{1\leq k\leq[\frac{m}{2}]+1}$ for which the following holds.
\begin{enumerate}[(i)]
\item $w_k(\cdot,x)\in C^{2,\alpha}(\R^n)$ uniformly for all $x\in\ov{\Omega}$ and $\norm{w_k(\cdot,x)}_{C^{2,\alpha}(\R^n)}\leq C_{m,k,g,\Omega}$.
\item $w_k(y,\cdot)\in C^{m-2k+2,1}(\ov{\Omega})$ uniformly for all $y\in\R^n$ and $\norm{w_k(y,\cdot)}_{C^{m-2k+2,1}(\ov{\Omega})}\leq C_{m,k,g,\Omega}$. Moreover, there holds for any $x_1,x_2\in\ov{\Omega}$ that
\begin{equation*}\label{eq:int cor-nl-x-C2a}
\begin{split}
\sum_{l=0}^{m-2k+1}&\norm{D_x^lw_k(\cdot,x_1)-D_x^lw_k(\cdot,x_2)}_{C^{2,\alpha}(\mbb{R}^n)}\leq C_{m,k,g,\Omega}|x_1-x_2|.
\end{split}
\end{equation*}
\item Provided that $k\geq 3$, for each $x\in\ov{\Omega}$, $w_k(\cdot,x)$ solves
\begin{equation}\label{eq:int cor-nl}
a_{ij}(\cdot,x)D_{y_iy_j}w_k(\cdot,x)+\Phi_k(\cdot,x)=0\quad\text{in }\mbb{R}^n,
\end{equation}
where
\begin{equation*}
\begin{split}
&\Phi_k=a_{ij}D_{x_ix_j}w_{k-2}+2a_{ij}D_{x_iy_j}w_{k-1}+\sum_{l=2}^{k-2}\frac{1}{l!}\sum_{n_1+\cdots+n_l=k-2}a_{i_1j_1\ldots i_lj_l}X_{i_1j_1}^{n_1}\cdots X_{i_lj_l}^{n_l},\\
&X_{i_rj_r}^{n_r}=D_{x_{i_r}x_{j_r}}w_{n_r}+2D_{x_{i_r}y_{j_r}}w_{n_r+1}+D_{y_{i_r}y_{j_r}}w_{n_r+2},\quad r=1,\ldots,l,\\
&a_{i_1j_1\ldots i_lj_l}=F_{p_{i_1j_1}\ldots p_{i_lj_l}}(D_x^2u+D_y^2w(\cdot;D_x^2u,\cdot),\cdot,\cdot),\quad l=1,\ldots,k-2.
\end{split}
\end{equation*}
\end{enumerate}
\end{lem} 

\begin{proof} We are going to use an induction argument to construct the desired family $\{w_k\}_{1\leq k\leq[\frac{m}{2}]+1}$ as well as families of functions $\{\psi_k:\ov{\Omega}\ra\mbb{R}\}_{-1\leq k\leq[\frac{m}{2}]+1}$ and $\{\phi_k:\mbb{R}^n\times\ov{\Omega}\ra\mbb{R}\}_{1\leq k\leq[\frac{m}{2}]+1}$, which verify the following conditions:
\begin{enumerate}
\item[(IP1)] $\phi_k(\cdot,x)\in C^{2,\alpha}(\R^n)$ uniformly for all $x\in\ov{\Omega}$ and $\norm{\phi_k(\cdot,x)}_{C^{2,\alpha}(\R^n)}\leq C_{m,k,g,\Omega}$.
\item[(IP2)] $\phi_k(y,\cdot)\in C^{m-2k+4,1}(\ov{\Omega})$ uniformly for $y\in\R^n$ and $\norm{\phi_k(y,\cdot)}_{C^{m-2k+4,1}(\ov{\Omega})}\leq\tilde{C}_{m,k,g,\Omega}$. Moreover, $\phi_k(0,\cdot)=0$ in $\ov{\Omega}$ and there holds for any $x_1,x_2\in\ov{\Omega}$ that
\begin{equation*}
\sum_{l=0}^{m-2k+3}\norm{D_x^l\phi_k(\cdot,x_1)-D_x^l\phi_k(\cdot,x_2)}_{C^{2,\alpha}(\R^n)}\leq \tilde{C}_{m,k,g,\Omega}|x_1-x_2|.
\end{equation*} 
\item[(IP3)] $\psi_k\in C^{m-2k+2,1}(\ov{\Omega})$ satisfying $\norm{\psi_k}_{C^{m-2k+2,1}(\ov{\Omega})}\leq\bar{C}_{m,k,g,\Omega}$.
\end{enumerate}
It will turn out at the end that as we define 
\begin{equation}\label{eq:int cor-nl-defn:wk}
w_k(y,x)=\phi_k(y,x)+\chi^{ij}(y,x)D_{x_ix_j}\psi_{k-2}(x)+\psi_k(x),
\end{equation}
where $\chi^{ij}(y,x):=D_{p_{ij}}w(y;D_x^2u,x)$, $\{w_k\}_{1\leq k\leq[\frac{m}{2}]+1}$ satisfies Lemma \ref{lem:int cor-nl}. 

Let us make a few remarks on the function $\chi^{ij}(y,x)$, which has the particular importance in this proof. First we observe from Proposition \ref{prop:reg-nl} and Lemma \ref{lem:u-nl-Cma} that $\chi^{ij}(\cdot,x)\in C^{2,\alpha}(\R^n)$ for all $x\in\ov{\Omega}$ and $\norm{\chi^{ij}(\cdot,x)}_{C^{2,\alpha}(\R^n)}\leq C_{m,g,\Omega}^{(1)}$. In addition, $\chi^{ij}(y,\cdot)\in C^{m-1,1}(\ov{\Omega})$ uniformly for $y\in\R^n$ and $\norm{\chi^{ij}(y,\cdot)}_{C^{m-1,1}(\ov{\Omega})}\leq C_{m,g,\Omega}^{(2)}$, and, in particular for $x_1,x_2\in\ov{\Omega}$, there holds
\begin{equation*}\label{eq:chi-nl-x-C2a}
\sum_{l=0}^{m-2}\norm{D_x^l\chi^{ij}(\cdot,x_1)-D_x^l\chi^{ij}(\cdot,x_2)}_{C^{2,\alpha}(\R^n)}\leq C_{m,g,\Omega}^{(2)}|x_1-x_2|.
\end{equation*}
It is noteworthy to see that, in view of the equation \eqref{eq:lin-nl}, $\chi^{ij}(\cdot,x)$ solves
\begin{equation}\label{eq:chi-nl}
a_{rs}(\cdot,x)D_{y_ry_s}\chi^{ij}(\cdot,x)+a_{ij}(\cdot,x)=\ov{a}_{ij}(x)\quad\text{in }\R^n,
\end{equation}
where $\ov{a}_{ij}(x)=\ov{F}_{p_{ij}}(D_x^2u,x)\in C^{m-1,1}(\ov{\Omega})$ whose $C^{m-1,1}(\ov{\Omega})$-norm is bounded above by $C_{m,g,\Omega}^{(2)}$. 

Let us now begin our induction argument. As the first step, we define $\psi_{-1}(x)=\psi_0(x)=\psi_{[\frac{m}{2}]}(x)=\psi_{[\frac{m}{2}]+1}(x)\equiv 0$ on $\ov{\Omega}$ and $\phi_1(y,x)\equiv0$, $\phi_2(y,x)=w(y;D_x^2u,x)$ on $\mbb{R}^n\times\ov{\Omega}$. If $m=2$ or 3, then $w_1(y,x)=0$ and $w_2(y,x)=w(y;D_x^2u,x)$, as we define them according to \eqref{eq:int cor-nl-defn:wk}. The assertions (i) and (ii) of Lemma \ref{lem:int cor-nl} are then immediate from Lemma \ref{lem:key-nl} and Proposition \ref{prop:reg-nl}. Since we have $k\leq 2$ when $m=2$ or 3, the assertion (iii) can be dismissed. Thus, Lemma \ref{lem:int cor-nl} is proved for the case $m=2$ and 3. 

Now we consider the case when $m\geq 4$. One can easily see that $\phi_1$ and $\phi_2$ [resp., $\psi_{-1}$, $\psi_0$, $\psi_{[\frac{m}{2}]}$ and $\psi_{[\frac{m}{2}]+1}$] chosen in the first step still verify (IP1)-(IP2) [resp., (IP3)].

In order to run the induction argument, we choose $3\leq k\leq[\frac{m}{2}]+1$ and suppose that we have already found the families $\{\psi_{l-2}\}_{1\leq l\leq k-1}$, $\{\phi_l\}_{1\leq l\leq k-1}$ and $\{w_l\}_{1\leq l\leq k-1}$ which satisfy (IP1)-(IP3) and Lemma \ref{lem:int cor-nl} respectively. We then define $\tilde{\Phi}_k:\mbb{R}^n\times\ov{\Omega}\ra\mbb{R}$ by
\begin{equation*}\label{eq:int cor-nl-defn:Phik}
\begin{split}
\tilde{\Phi}_k&=a_{ij}D_{x_ix_j}(\phi_{k-2}+\chi^{ab}D_{x_ax_b}\psi_{k-4})+2a_{ij}D_{x_iy_j}(\phi_{k-1}+\chi^{ab}D_{x_ax_b}\psi_{k-3})\\
&\quad+\sum_{l=2}^{k-2}\frac{1}{l!}\sum_{n_1+\cdots+n_l=k-2}a_{i_1j_1\ldots i_lj_l}X_{i_1j_1}^{n_1}\cdots X_{i_lj_l}^{n_l}.
\end{split}
\end{equation*}
One may notice that $\tilde{\Phi}_k$ does not involve the functions $\psi_{r-2}$ and $\phi_r$ for $r\geq k$. 

Consider the following problem: For each $x\in\ov{\Omega}$, there exists a unique constant $\Psi_{k-2}(x)$ such that the following PDE,
\begin{equation}\label{eq:int cor-nl-1}
a_{ij}(\cdot,x)D_{y_iy_j}v+\tilde{\Phi}_k(\cdot,x)=\Psi_{k-2}(x)\quad\text{in }\R^n,
\end{equation}
attains a bounded 1-periodic solution $v$. Note that $a_{ij}(\cdot,x)$ is uniformly elliptic with the ellipticity constants $\lambda$ and $\Lambda$. Moreover, $a_{i_1j_1\ldots i_lj_l}(\cdot,x)$ is 1-periodic and belongs to $C^{m-l,1}(\R^n)$ whose $C^{m-l,1}(\R^n)$-norm is bounded above by a constant $K_{m,l,g,\Omega}$. This fact together with our induction hypotheses, (IP1)-(IP3) and Lemma \ref{lem:int cor-nl} (i) and (ii), yields that $\tilde{\Phi}_k(\cdot,x)\in C^{0,\alpha}(\R^n)$ where its $C^{0,\alpha}(\R^n)$-norm is bounded above by a constant $\tilde{K}_{m,k,g,\Omega}$. Therefore, Lemma \ref{lem:key-l} yields that the PDE \eqref{eq:int cor-nl-1} is solvable with a $C^{2,\alpha}(\R^n)$-solution, and denote it by $\phi_k(\cdot,x)$. In particular, let us choose $\phi_k(\cdot,x)$ such that $\phi_k(0,x)=0$. Since the domain $\Omega$ is bounded, $\phi_k(\cdot,x)\in C^{2,\alpha}(\R^n)$ uniformly for $x\in\ov{\Omega}$ and $\norm{\phi_k(\cdot,x)}_{C^{2,\alpha}(\R^n)}\leq C_{m,k.g,\Omega}$. Therefore, $\phi_k$ verifies (IP1). 

To know the regularity of $\phi_k$ in $x$-variable, we utilize Proposition \ref{prop:reg-nl}. We know that $a_{i_1j_1\ldots i_mj_m}(y,\cdot)\in C^{m-l,1}(\ov{\Omega})$ and its $C^{m-l,1}(\ov{\Omega})$-norm is bounded above by $L_{m,k,g,\Omega}$. Then again by using our induction hypotheses, we obtain $\tilde{\Phi}_k(y,\cdot)\in C^{m-2k+4,1}(\ov{\Omega})$ whose $C^{m-2k+4,1}(\ov{\Omega})$-norm is bounded above by $\tilde{L}_{m,k,g,\Omega}$. Thus, Proposition \ref{prop:reg-nl} implies that both $\Psi_{k-2}$ and $\phi_k(y,\cdot)$ belong to $C^{m-2k+4,1}(\ov{\Omega})$ with the estimate that $\max\{\norm{\Psi_{k-2}}_{C^{m-2k+4,1}(\ov{\Omega})},\norm{\phi_k(y,\cdot)}_{C^{m-2k+4,1}(\ov{\Omega})}\}\leq \tilde{C}_{m,k,g,\Omega}$; in particular, we obtain for any $x_1,x_2\in\ov{\Omega}$ that
\begin{equation*}\label{eq:int cor-nl-1-x-C2a}
\sum_{i=0}^{m-2k+3}\norm{D_x^i\phi_k(\cdot,x_1)-D_x^i\phi_k(\cdot,x_2)}_{C^{2,\alpha}(\R^n)}\leq\tilde{C}_{m,k,g,\Omega}|x_1-x_2|.
\end{equation*}
Hence, $\phi_k$ satisfies (IP2) as well.

Moreover, we choose the function $\psi_{k-2}:\ov{\Omega}\ra\mbb{R}$ by the solution of 
\begin{equation}\label{eq:int cor-nl-2}
\begin{cases}
\ov{a}_{ij}D_{x_ix_j}\psi_{k-2}=-\Psi_{k-2}&\text{in }\Omega,\\
\psi_{k-2}=0&\text{on }\p\Omega.
\end{cases}
\end{equation}
Recall from Lemma \ref{lem:prop:nl} that $\bar{a}_{ij}$ is uniformly elliptic in $\ov{\Omega}$ with the ellipticity constants $\lambda$ and $\Lambda$. Also Proposition \ref{prop:reg-nl} implies that $\bar{a}_{ij}\in C^{m-1,1}(\ov{\Omega})$ whose $C^{m-1,1}(\ov{\Omega})$-norm is bounded above by $C_{m,g,\Omega}$. Since $\Psi_{k-2}\in C^{m-2k+4,1}(\ov{\Omega})$, there exists a unique solution $\psi_{k-2}\in C^{m-2k+6,1}(\ov{\Omega})$ of \eqref{eq:int cor-nl-2} and 
\begin{equation*}\label{eq:int cor-nl-2-x}
\norm{\psi_{k-2}}_{C^{m-2k+6,1}(\ov{\Omega})}\leq C_{\norm{\bar{a}_{ij}}_{C^{m-1,1}(\ov{\Omega})},\Omega}(\norm{\psi}_{L^\infty(\Omega)}+\norm{\Psi}_{C^{m-2,1}(\ov{\Omega})})\leq\bar{C}_{m,k-2,g,\Omega}.
\end{equation*}
Thus, $\psi_{k-2}$ satisfies the induction hypothesis (IP3).

Define $v_k:\mbb{R}^n\times\ov{\Omega}\ra\mbb{R}$ by 
\begin{equation*}
v_k(y,x):=\phi_k(y,x)+\chi^{ij}(y,x)D_{x_ix_j}\psi_{k-2}(x). 
\end{equation*}
It then follows from the observations above that $v_k(\cdot,x)\in C^{2,\alpha}(\R^n)$ with the estimate $\norm{v_k(\cdot,x)}_{C^{2,\alpha}(\R^n)}\leq A_{m,k,g,\Omega}$ and that $v_k(y,\cdot)\in C^{m-2k+2,1}(\ov{\Omega})$ with the estimate $\norm{v_k(y,\cdot)}_{C^{m-2k+2,1}(\ov{\Omega})}\leq \tilde{A}_{m,k,g,\Omega}$; furthermore, we have for any pair of $x_1,x_2\in\ov{\Omega}$ that
\begin{equation*}
\sum_{i=0}^{m-2k+1}\norm{D_x^iv_k(\cdot,x_1)-D_x^iv_k(\cdot,x_2)}_{C^{2,\alpha}(\R^n)}\leq\tilde{A}_{m,k,g,\Omega}|x_1-x_2|.
\end{equation*}
One may also check that $A_{m,k,g,\Omega}=C_{m,k,g,\Omega}+C_{m,g,\Omega}^{(1)}\bar{C}_{m,k-2,g,\Omega}$ and $\tilde{A}_{m,k,g,\Omega}=\tilde{C}_{m,k,g,\Omega}+C_{m,g,\Omega}^{(2)}\bar{C}_{m,k-2,g,\Omega}$. Moreover, we combine \eqref{eq:int cor-nl-1} and \eqref{eq:int cor-nl-2} and obtain that
\begin{equation*}
\begin{split}
&a_{ij}(\cdot,x)D_{y_iy_j}v_k(\cdot,x)+\Phi_k(\cdot,x)\\
&=a_{ij}(\cdot,x)D_{y_iy_j}\phi_k(\cdot,x)+\tilde{\Phi}_k(\cdot,x)+[a_{rs}(\cdot,x)D_{y_ry_s}\chi^{ij}(\cdot,x)+a_{ij}(\cdot,x)]D_{x_ix_j}\psi_{k-2}(x)\\
&=\Psi_{k-2}(x)+\ov{A}_{ij}D_{x_ix_j}\psi_{k-2}(x)\\
&=0\quad\text{in }\mbb{R}^n.
\end{split}
\end{equation*} 
Hence, $v_k$ satisfies Lemma \ref{lem:int cor-nl}.

We have obtained so far $\psi_{k-2}$, $\phi_k$ and $v_k$ which satisfy (IP1)-(IP3) and Lemma \ref{lem:int cor-nl} respectively. Now we apply the same argument above using 
\begin{equation*}\label{eq:int cor-nl-defn:Phik}
\begin{split}
\hat{\tilde{\Phi}}_{k+1}&=a_{ij}D_{x_ix_j}(\phi_{k-1}+\chi^{ab}D_{x_ax_b}\psi_{k-3})+2a_{ij}D_{x_iy_j}(\phi_k+\chi^{ab}D_{x_ax_b}\psi_{k-2})\\
&\quad+\sum_{l=2}^{k-1}\frac{1}{l!}\sum_{n_1+\cdots+n_l=k-1}a_{i_1j_1\ldots i_lj_l}\hat{X}_{i_1j_1}^{n_1}\cdots\hat{X}_{i_lj_l}^{n_l},
\end{split}
\end{equation*}
where $\hat{X}_{i_rj_r}^l=X_{i_rj_r}^l$ for $1\leq l\leq k-3$ and $\hat{X}_{i_rj_r}^{k-2}=D_{x_{i_r}x_{j_r}}w_{k-2}+2D_{x_{i_r}y_{j_r}}w_{k-1}+D_{y_{i_r}y_{j_r}}v_k$. Then we obtain $\psi_{k-1}$, $\phi_{k+1}$ and $v_{k+1}$ which satisfy (IP1)-(IP3) and Lemma \ref{lem:int cor-nl} respectively. Applying the same argument once again using
\begin{equation*}\label{eq:int cor-nl-defn:Phik}
\begin{split}
\hat{\tilde{\Phi}}_{k+2}&=a_{ij}D_{x_ix_j}(\phi_k+\chi^{ab}D_{x_ax_b}\psi_{k-2})+2a_{ij}D_{x_iy_j}(\phi_{k+1}+\chi^{ab}D_{x_ax_b}\psi_{k-1})\\
&\quad+\sum_{l=2}^k\frac{1}{l!}\sum_{n_1+\cdots+n_l=k}a_{i_1j_1\ldots i_lj_l}\hat{X}_{i_1j_1}^{n_1}\cdots\hat{X}_{i_lj_l}^{n_l},
\end{split}
\end{equation*}
where $\hat{X}_{i_rj_r}^l=X_{i_rj_r}^l$ for $1\leq l\leq k-3$, $\hat{X}_{i_rj_r}^{k-2}=D_{x_{i_r}x_{j_r}}w_{k-2}+2D_{x_{i_r}y_{j_r}}w_{k-1}+D_{y_{i_r}y_{j_r}}v_k$ and $\hat{X}_{i_rj_r}^{k-1}=D_{x_{i_r}x_{j_r}}w_{k-1}+2D_{x_{i_r}y_{j_r}}v_k+D_{y_{i_r}y_{j_r}}v_{k+1}$, we get $\psi_k$, $\phi_{k+2}$ and $v_{k+2}$ satisfying (IP1)-(IP3) and Lemma \ref{lem:int cor-nl} respectively. 

Now let us define $w_k$ as in \eqref{eq:int cor-nl-defn:wk}; i.e., $w_k(y,x)=v_k(y,x)+\psi_k(x)$. Then $w_k$ satisfies Lemma \ref{lem:int cor-nl}; in particular, the estimates are satisfied with the constant $\max\{A_{m,k,g,\Omega}+\tilde{A}_{m,k,g,\Omega}\}+\bar{C}_{m,k,g,\Omega}$. In addition, one can check that 
\begin{equation*}\label{eq:int cor-nl-defn:Phik}
\begin{split}
\hat{\tilde{\Phi}}_{k+1}&=a_{ij}D_{x_ix_j}(\phi_{k-1}+\chi^{ab}D_{x_ax_b}\psi_{k-3})+2a_{ij}D_{x_iy_j}(\phi_k+\chi^{ab}D_{x_ax_b}\psi_{k-2})\\
&\quad+\sum_{l=2}^{k-1}\frac{1}{l!}\sum_{n_1+\cdots+n_l=k-1}a_{i_1j_1\ldots i_lj_l}X_{i_1j_1}^{n_1}\cdots X_{i_lj_l}^{n_l}\\
&=:\tilde{\Phi}_{k+1},
\end{split}
\end{equation*}
which implies that the functions $\psi_{k-1}$ and $\phi_{k+1}$ are not changed by replacing $v_k$ by $w_k$ in the induction argument. Therefore, our induction argument runs through $k=3,\cdots,[\frac{m}{2}]+1$, by which we obtain the families $\{\psi_{k-2}\}_{1\leq k\leq[\frac{m}{2}]+1}$, $\{\phi_k\}_{1\leq k\leq[\frac{m}{2}]+1}$ and $\{w_k\}_{1\leq k\leq[\frac{m}{2}]+1}$, where $w_{[\frac{m}{2}]}=v_{[\frac{m}{2}]}$ and $w_{[\frac{m}{2}]+1}=v_{[\frac{m}{2}]+1}$. Recall that we have chosen $\psi_{[\frac{m}{2}]}=\psi_{[\frac{m}{2}]+1}\equiv 0$. Thus, we have constructed all the desired families $\{\psi_k\}_{-1\leq k\leq[\frac{m}{2}]+1}$, $\{\phi_k\}_{1\leq k\leq[\frac{m}{2}]+1}$ and $\{w_k\}_{1\leq k\leq[\frac{m}{2}]+1}$ which satisfy (IP1)-(IP3) and Lemma \ref{lem:int cor-nl} respectively. It completes our proof.
\end{proof} 

\begin{rem} As we note in the remark below Proposition \ref{prop:reg-nl}, we see how the coupling effect contribute to the regularity of $x\mapsto w_k(y,x)$. If the $x$ and $y$-variables were decoupled, we would have obtained $w_k(\cdot,x)\in C^{m-k+2,1}(\ov{\Omega})$.
\end{rem}

To this end we define the $k$-th order interior corrector $w_k^\e$ of \eqref{eq:main-l} for each $1\leq k\leq[\frac{m}{2}]+1$ and $\e>0$ by 
\begin{equation}\label{eq:int cor-nl-form-e}
w_k^\e(x)=w_k\left(\frac{x}{\e},x\right)\quad(x\in\ov{\Omega}),
\end{equation}
where $w_k$'s are given in accordance with Lemma \ref{lem:int cor-nl}, and define $\eta_m^\e:\ov{\Omega}\ra\mbb{R}$ by
\begin{equation}\label{eq:int cor-nl-eta}
\eta_m^\e=u+\e w_1^\e+\cdots+\e^{[\frac{m}{2}]+1}w_{[\frac{m}{2}]+1}^\e.
\end{equation}

Now we are in position to introduce the boundary layer corrector. The underlying idea of seeking the boundary layer corrector is the same as in the linear case; we correct the boundary oscillation occurred by the interior correctors by solving the corresponding boundary value problem (c.f. \eqref{eq:bdry cor-l}). Due to the nonlinearity of the problem \eqref{eq:nl}, however, we cannot find the boundary layer corrector in an order-wise manner. Instead, we consider a boundary value problem which involves the entire boundary oscillation caused by the interior correctors; i.e., we solve for each $\e>0$ the following PDE,
\begin{equation}\label{eq:bdry cor-nl}
\begin{cases}
F\left(D^2\eta_m^\e+D^2\theta_m^\e,x,\e^{-1}x\right)=F\left(D^2\eta_m^\e,x,\e^{-1}x\right)&\text{in }\Omega,\\
\theta_m^\e=-\eta_m^\e+g&\text{on }\p\Omega.
\end{cases}
\end{equation}
One may notice from Lemma \ref{lem:int cor-nl} that $\eta_m^\e\in C^2(\ov{\Omega})$ that the right hand side of \eqref{eq:bdry cor-nl} is a uniformly continuous function on $\ov{\Omega}$ for each $\e>0$. Thus, Perron's method (e.g., Theorem \ref{thm:exist-nl-apndx}) ensures the unique existence of a viscosity solution $\theta_m^\e\in C(\ov{\Omega})$ of \eqref{eq:bdry cor-nl}.

\subsection{Proof of Main Theorem II}\label{sec:pf main-nl}

We shall now prove Main Theorem II.

\begin{proof}[Proof of Theorem \ref{thm:main-nl}]  Suppose that $m\geq 4$. The first part of the proof verifies the discussion we made in the beginning of the previous subsection. Fix $\e_*\in(0,1)$ and pick any $\e>0$. We will skip the calculation if it has already been done in the previous subsection. 

In what follows let us denote by $r_m$ the positive integer $[\frac{m}{2}]+1$. We choose the family $\{w_k\}_{1\leq k\leq r_m}$ from Lemma \ref{lem:int cor-nl}. Next we define the family $\{X^k\}_{1\leq k\leq r_m}$ as in \eqref{eq:Xk-nl} and then the function $Y^{r_m}$ as in \eqref{eq:Yr-nl}. By Lemma \ref{lem:int cor-nl} (i)-(ii), we have a uniform bound on the matrix norm of $X^k$, which is independent of $\e$, namely,
\begin{equation}\label{eq:main-nl-linf-Xk}
\norm{X^k(\cdot/\e,\cdot)}_{L^\infty(\Omega)}\leq C_{m,k,g,\Omega}.
\end{equation}
It is then immediately follows that
\begin{equation}\label{eq:main-nl-linf-Ym}
\sup_{0<\e\leq\e_*}\norm{Y^{r_m}(\cdot/\e,\cdot)}_{L^\infty(\ov{\Omega})}\leq(1-\e_*)L_*\frac{1-\e^{r_m}}{1-\e}<L_*
\end{equation}
where $L_*=(1-\e_*)^{-1}\max\{1,C_{m,1,g,\Omega},\ldots,C_{m,r_m,g,\Omega}\}$. 

In the rest of this proof, we set $\e\in(0,\e_*]$ to be fixed. We choose any $x\in\Omega$ and adopt the Taylor expansion of $F(D^2\eta_m^\e,x,x/\e)$ with respect to the $M$-variable up to $(r_m-1)$-th order. For brevity, we omit the dependency on $(\e^{-1}x,x)$. Then, by the choice of our interior correctors $w_k^\e$, we end up with
\begin{equation}\label{eq:main-nl-te}
\begin{split}
F(D^2\eta_m^\e)&=F(X^0+\e Y^{r_m})\\
&=F(X^0)+\sum_{k=1}^{r_m-1}\frac{\e^k}{k!}F_{p_{i_1j_1}\ldots p_{i_kj_k}}(X^0)Y^{r_m}_{i_1j_1}\cdots Y^{r_m}_{i_kj_k}+R_m^\e\\
&=F(X^0)+\sum_{k=1}^{r_m-1}\e^k\sum_{l=1}^k\frac{1}{l!}\sum_{n_1+\cdots+n_l=k}F_{p_{i_1j_1}\ldots p_{i_lj_l}}(X^0)X^{n_1}_{i_1j_1}\cdots X^{n_l}_{i_lj_l}+\tilde{R}_m^\e\\
&=\tilde{R}_m^\e,
\end{split}
\end{equation}
where 
\begin{equation*}
\begin{split}
R_m^\e&=\frac{\e_0^{r_m}}{r_m!}F_{p_{i_1j_1}\ldots p_{i_{r_m}i_{r_m}}}(X^0)Y_{i_1j_1}^{r_m}\cdots Y_{i_{r_m}j_{r_m}}^{r_m}\quad\text{for some }\e_0\in[0,\e],\\
\tilde{R}_m^\e&=R_m^\e+\sum_{k=1}^{r_m-2}\sum_{r_m-1\leq n_1+\cdots+n_k\leq r_mk}\frac{\e^{n_1+\cdots+n_k}}{k!}F_{p_{i_1j_1}\ldots p_{i_kj_k}}(X^0)X^{n_1}_{i_1j_1}\cdots X^{n_k}_{i_kj_k}.
\end{split}
\end{equation*} 
One should note that $F_{p_{i_1j_1}\ldots p_{i_kj_k}}(X^0)$ are exactly the coefficients $a_{i_1j_1\ldots i_kj_k}$ appearing in \eqref{eq:int cor-nl}. Now due to \eqref{eq:main-nl-linf-Xk} and \eqref{eq:main-nl-linf-Ym}, there hold $|R_m^\e|\leq\tilde{C}_{m,g,\Omega}L_*^{r_m}\e^{r_m}$, and thus, 
\begin{equation*}\label{eq:main-nl-linf-tRm}
|\tilde{R}_m^\e|\leq|R_m^\e|+\tilde{\tilde{C}}_{m,g,\Omega}L_*^{(r_m-2)r_m}\e^{r_m-1}\leq C_0\e^{r_m-1}.
\end{equation*}

The second part of this proof is devoted to the establishment of the estimate \eqref{eq:main-nl}. The essence is to construct barriers and argue by the comparison principle. Choose $R>0$ in such a way that $\ov{\Omega}\subset B_R(0)$. Consider the functions $\eta_m^{\e,\pm}:\ov{\Omega}\ra\mbb{R}$ defined by
\begin{equation}\label{eq:main-nl-defn:barrier}
\eta_m^{\e,\pm}=\eta_m^\e+\theta_m^\e\pm(2\lambda)^{-1}C_0\e^{r_m-1}(R^2-|x|^2)\quad(x\in\ov{\Omega}).
\end{equation}
By the uniform ellipticity of $F$ (structure condition (F2)) and the choice of the boundary layer corrector \eqref{eq:bdry cor-nl}, there holds 
\begin{equation*}
\begin{split}
F(D^2\eta_m^{\e,+})&\leq F(D^2\eta_m^\e+D^2\theta_m^\e)-C_0\e^{r_m-1}=F(D^2\eta_m^\e)-C_0\e^{r_m-1}\leq 0
\end{split}
\end{equation*}
in the viscosity sense, and $\eta_m^{\e,+}|_{\p\Omega}\geq\eta_m^\e+\theta_m^\e=g$. Thus, $\eta_m^{\e,+}$ is a viscosity supersolution of \eqref{eq:nl}. In a similar manner, one can verify that $\eta_m^{\e,-}$ is a viscosity subsolution of \eqref{eq:nl}. Thus, the comparison principle yields $\eta_m^{\e,-}\leq u^\e\leq\eta_m^{\e,+}\quad\text{in }\ov{\Omega}$. It then follows that
\begin{equation*}
\norm{u^\e-\eta_m^\e-\theta_m^\e}_{L^\infty(\Omega)}\leq(2\lambda)^{-1}C_0\e^{r_m-1},
\end{equation*}
which proves \eqref{eq:main-nl}.

The proof for the case $m=2$ or 3 shares the same idea presented above, but is simpler. In this case, $\eta_m^\e(x)=u(x)+\e^2w_2(\e^{-1}x,x)$, and thus, we do not need the expansion \eqref{eq:main-nl-te}; instead we can directly argue as in the second part. The rest of the proof is exactly the same, so is omitted.  
\end{proof}

\appendix
\section{Existence and Regularity Theory of Uniformly Elliptic Equations}\label{sec:nl-apndx}

Set $\Omega$ to be a bounded domain of $\R^n$ and $F\in C(\mc{S}^n\times\R^n)$ be uniformly elliptic with ellipticity constants $0<\lambda\leq\Lambda$. Also let $f\in C(\R^n)$. The notion of viscosity solutions can be found in many literatures; e.g., see \cite{CIL} and \cite{CC}. 

\begin{thm}\label{thm:exist-nl-apndx} Suppose that $F$ is 1-periodic in $y$ and $\delta>0$. 
\begin{enumerate}[(a)]
\item (Comparison principle) If $u$ and $v$ are respectively 1-periodic viscosity sub- and super-solution of $F(D^2w,y)-\delta w=0$ in $\R^n$, then $u\leq v$ in $\R^n$. 
\item (Perron's method) If such $u$ and $v$ in (a) exist, then there exists a unique 1-periodic viscosity solution $w$ such that $u\leq w\leq v$ in $\R^n$. 
\end{enumerate}
\end{thm}

\begin{thm}\label{thm:reg-nl-apndx}
\begin{enumerate}[(a)]
\item (Harnack inequality) If $u\in S^*(\lambda,\Lambda,f)$ and $u\geq 0$ in $Q_1$, then $\sup_{Q_{1/2}}u\leq C(\inf_{Q_{1/2}}u+\norm{f}_{L^n(Q_1)})$ for a universal $C>0$.
\item (Interior $C^\alpha$-regularity) If $u\in S^*(\lambda,\Lambda,f)$ in $Q_1$, then $u\in C^\alpha(\ov{Q}_{1/2})$ and $\norm{u}_{C^\alpha(\ov{Q}_{1/2})}\leq C(\norm{u}_{L^\infty(Q_1)}+\norm{f}_{L^\infty(Q_1)})$ for universal $\alpha\in(0,1)$ and $C>0$. 
\item (Liouville theorem) Any bounded below (or above) function in $S(\lambda,\Lambda,0)$ on $\mbb{R}^n$ is constant.
\item (Modulus of continuity) Suppose $u\in S(\lambda,\Lambda,f)$ in $\Omega$, and there is a modulus of continuity $\rho$ of $\vp:=u|_{\p\Omega}$. If $\Omega$ satisfy the uniform exterior sphere condition with radius $R$, then there exists a modulus of continuity $\rho^*$ of $u$ in $\ov{\Omega}$, where $\rho$ depends only on $n$, $\lambda$, $\Lambda$, $\diam(\Omega)$, $R$, $\norm{\vp}_{L^\infty(\Omega)}$, and $\norm{f}_{L^\infty(\Omega)})$.
\item (Interior $C^{2,\alpha}$-regularity) Suppose that $F$ is concave in $M$ and $F(0,0)=f(0)$. Define $\beta(x):=\sup_{M\in\mc{S}^n}\frac{|F(M,)-F(M,0)|}{\norm{M}+1}$. If $\beta,f\in C^\alpha(B_2)$ for some $\alpha\in(0,1)$, and if $F(D^2u,x)=f(x)$ in $B_2$, then $u\in C^{2,\alpha}(\ov{B}_{2/C})$ and $\norm{u}_{C^{2,\alpha}(\ov{B}_{2/C})}^*\leq C(\norm{u}_{L^\infty(B_2)}+\norm{f}_{C^\alpha(B_2)}+1)$, for some $C>1$ depends only on $n$, $\lambda$, $\Lambda$, $\alpha$, and $\norm{\beta}_{C^\alpha(B_2)}$.
\item ($C^{k+2,\alpha}$-regularity) Suppose that $F\in C^{m,1}(\mc{S}^n\times\ov{\Omega})$ and $f\in C^{m,1}(\ov{\Omega})$. If $u\in C^{2,\alpha}(\Omega)$ solves $F(D^2u,x)=f(x)$, then $\p\Omega\in C^{m+2,1}$. Moreover, if $\partial\Omega\in C^{m+2,1}$  and $u|_{\p\Omega}\in C^{m+2,1}$, then $u\in C^{m+2,\alpha}(\Omega)$. 
\end{enumerate}
\end{thm}

\begin{acknowledgements}
Sunghan Kim was supported by NRF(National Research Foundation of Korea) Grant funded by the Korean Government(NRF-2014-Fostering Core Leaders of the Future Basic Science Program). 

Ki-Ahm Lee  was supported by the National Research Foundation of Korea(NRF) grant funded by the Korea government(MSIP) (No.2014R1A2A2A01004618). 
Ki-Ahm Lee also hold a joint appointment with the Research Institute of Mathematics of Seoul National University.

The authors would like to thank for anonymous reviewers for their valuable comments and suggestions. 
\end{acknowledgements}

\end{document}